\documentclass[a4paper,reqno]{amsart}
\usepackage{amssymb, amsfonts, amsthm,amsmath,amsaddr}
\usepackage{tikz}
\usepackage{bbm}
\usepackage{enumerate}
\usepackage{fullpage}

\newtheorem{theorem}{Theorem}
\newtheorem{lemma}[theorem]{Lemma}
\newtheorem{proposition}[theorem]{Proposition}

\newcommand{\eps}{\varepsilon}

\newcommand{\ex}{\mathrm{ex}}
\newcommand{\T}[2]{T_{#1}(#2)}

\newcommand{\partition}[1]{v_{#1}}
\newcommand{\badedge}{b}
\newcommand{\badedgep}[1]{\badedge_{#1}}

\begin{document}
\title{Supersaturation Problem for the Bowtie}
\author{Mihyun Kang\textsuperscript{1}, Tam\'as Makai\textsuperscript{1}}
\email{\{kang,makai\}@math.tugraz.at}
\address{Institute of Discrete Mathematics\\ Graz University of Technology\\8010 Graz, Austria}
\thanks{\textsuperscript{1} Supported by Austrian Science Fund (FWF): P26826}
	
\author[M. Kang, T. Makai, and O. Pikhurko]{Oleg Pikhurko\textsuperscript{2}}	
\email{O.Pikhurko@warwick.ac.uk}
\address{Mathematics Institute and DIMAP\\ University of Warwick\\ Coventry CV4 7AL, United Kingdom}	
\thanks{\textsuperscript{2} Supported by ERC: 306493 and EPSRC: EP/K012045/1}

\begin{abstract}
	The Tur\'an function $\ex(n,F)$ denotes the maximal number of edges in an $F$-free graph on $n$ vertices.  We consider the function $h_F(n,q)$, the minimal number of copies of $F$ in a graph on $n$ vertices with $\ex(n,F)+q$ edges. The value of $h_F(n,q)$ has been extensively studied when $F$ is bipartite or colour-critical. In this paper we investigate the simplest remaining graph $F$, namely, two triangles sharing a vertex, and establish the asymptotic value of $h_F(n,q)$ for $q=o(n^2)$.
\end{abstract}

\maketitle

\section{Introduction}

The \emph{Tur\'an function} $\ex(n,F)$ of a graph $F$ is the maximum number of edges in an $F$-free graph on $n$ vertices. 
In 1907, Mantel \cite{mantel:07} proved that $\ex(n,K_3)=\lfloor n^2/4\rfloor$, where $K_r$ denotes the complete graph on $r$ vertices.
The fundamental paper of Tur\'an~\cite{turan:41} solved this extremal problem for cliques: the \emph{Tur\'an graph} $\T{r}{n}$, the complete $r$-partite graph on $n$ vertices with parts of size $\lfloor n/r \rfloor$ or $\lceil n/r \rceil$, is the unique maximum $K_{r+1}$-free graph on $n$ vertices. Thus the Tur\'an function satisfies $\ex(n,K_{r+1})=|E(\T{r}{n})|$.

Stated in the contrapositive, this implies that a graph with $\ex(n,K_{r+1})+1$ edges (where, by default, $n$ denotes the number of vertices) contains at least one copy of $K_{r+1}$. Rademacher (1941, unpublished) showed that a graph with $\lfloor n^2/4 \rfloor + 1$ edges contains not just one but at least $\lfloor n/2 \rfloor$ copies of a triangle. This is perhaps the first result in the so-called ``theory of supersaturated graphs'' that focuses on the \emph{supersaturation function} 
$$h_F(n,q)=\min\{\# F(H):|V(H)|=n, |E(H)|=\ex(n,F)+q\},$$
the minimum number of $F$-subgraphs in a graph $H$ on $n$ vertices and $\ex(n,F)+q$ edges. (We say that $G$ is a \emph{subgraph} of $H$ if $V(G)\subseteq V(H)$ and $E(G)\subseteq E(H)$; we call $G$ an \emph{$F$-subgraph} if it is isomorphic to~$F$.)

One possible construction for graphs with the minimal number of copies of $F$ is to add some $q$ edges to a maximum $F$-free graph. Denote by $t_F(n,q)$ the smallest number of $F$-subgraphs that can be achieved this way. Clearly, $h_F(n,q) \leq t_F(n,q)$. In fact, this bound is sharp for cliques, when $q$ is small. Erd\H os~\cite{erdos:55} extended Rademacher's result by showing that $h_{K_3}(n,q)=t_{K_3}(n,q)=q\lfloor n/2\rfloor$ for $q\le 3$. Later, he \cite{erdos:62b}  
showed that there exists some small constant $\varepsilon_{K_r} > 0$ such that $h_{K_{r}}(n,q)=t_{K_{r}}(n,q)$ for all $q\le \varepsilon_{K_r} n$.
Lov\'{a}sz and Simonovits \cite{lovasz+simonovits:75,lovasz+simonovits:83} found the best possible value of $\varepsilon_{K_r}$ as $n\to\infty$, settling a long-standing conjecture of Erd\H os~\cite{erdos:55}. In fact, the second paper~\cite{lovasz+simonovits:83} completely solved the $h_{K_r}(n,q)$-problem when $q = o(n^2)$. The case $q=\Omega(n^2)$ of the supersaturation problem for cliques has been actively studied and proved notoriously difficult. Only recently was an asymptotic solution found: by Razborov \cite{razborov:2008} for $K_3$ (see also Fisher~\cite{fisher:89}), by Nikiforov \cite{nikiforov:2011} for $K_4$, and by Reiher~\cite{reiher:Kr} for general $K_r$.

The supersaturation problem was also considered for general graphs $F$.  If $F$ is bipartite, then there is a beautiful (and still open) 
conjecture of 
Erd\H os--Simonovits (see~\cite{Simonovits84}) and 
Sidorenko~\cite{Sidorenko93} whose positive solution would 
determine $h_F(n,q)$ asymptotically for
$q=\Omega(n^2)$. We refer the reader to some
recent papers on the topic, \cite{ConlonFoxSudakov10,Hatami10,KimLeeLee16,LiSzegedy11arxiv,Szegedy14arxiv:v3}, that contain many
references.

For non-bipartite $F$, the value of $h_F(n,q)$ has also been considered for general colour-critical graphs. A graph is called \emph{$r$-(colour)-critical} if its chromatic number is $r+1$ while the removal of some edge from the graph reduces its chromatic number. Simonovits \cite{simonovits:68} established that if $F$ is $r$-critical, then the unique maximal $F$-free graph is $T_r(n)$. Pikhurko and Yilma \cite{pikhurko+yilma:16} extending the results of Mubayi \cite{mubayi:2010} established that, similarly to cliques, for every colour-critical graph $F$ there exists $\varepsilon_F>0$ such that when $q\leq \varepsilon_F n$, we have $h_F(n,q)=t_F(n,q)$. In addition, they established the asymptotic size of $h_F(n,q)$ when $q=o(n^2)$.

As far as we know, the supersaturation problem has not been considered for graphs which have chromatic number at least 3 and are not colour-critical, apart from some general (and rather weak) bounds by Erd\H{o}s and Simonovits \cite{ErdosSimonovits83}. 
In this paper, we investigate a `simplest'  such graph, namely the \emph{bowtie} which consists of two copies of $K_3$ merged at a vertex. We refer to this vertex as the \emph{central vertex} of the bowtie.

Despite the simple nature of the bowtie, bowtie-free graphs have been crucial in several areas, such as in countable universal graphs \cite{MR1683298,MR1675931}, in Ramsey theory \cite{MR3767511} and Hrushovski property \cite{Evans17}, etc. 

{\bf From this point on $F$ denotes the bowtie and we will assume that $n$ is sufficiently large.} The main contribution of this paper is twofold. First we establish that when $q=o(n^2)$ any graph on $n$ vertices and $\ex(n,F)+q$ edges with minimal number of copies of $F$ contains the Tur\'an graph $T_2(n)$. Based on this we establish the asymptotic value of $h_F(n,q)$ when $q=o(n^2)$.

\section{Main results}

The Tur\'an function of the bowtie $F$ is known, namely 
$\ex(n,F)=\lfloor n^2/4 \rfloor+1$ for $n\ge 5$. In addition, the extremal graphs are also known, each being $T_2(n)$ with an arbitrary edge added. (This is apparently a folklore result; the easiest proof known to us proceeds by removing a vertex of minimum degree and using induction on~$n$.) When $n$ is odd, this leads to two non-isomorphic versions of the extremal graph. 
Our first result shows that, when $n$ is large enough and $q=o(n^2)$, the only way to construct a graph containing as few bowties as possible is to add edges to $T_2(n)$.

\begin{theorem}\label{thm:subgraph}
	There is $\delta>0$ such that every graph $H$ with $n\ge 1/\delta$ vertices, $\ex(n,F)+q\le (1/4+\delta)n^2$ edges and $h_F(n,q)$ copies of the bowtie $F$ contains $T_2(n)$ as a subgraph.
\end{theorem}

Theorem~\ref{thm:subgraph} implies that $t_F(n,q)=h_{F}(n,q)$ when $q=o(n^2)$. Once this structural property of the graph has been established, we deduce the asymptotic number of bowties present.

\begin{theorem}\label{thm:assymptotics}
	For every $c>0$ there is $\delta>0$ such that for all natural numbers $n\ge 1/\delta$ and $q\le \delta n^2$ the following holds. If we write $2(q+1)=dn+m$ for $d,m\in\mathbb{N}$ with $0\le m<n$ and set 
	$$e_1=\left\lfloor \frac{dn}{4}+\frac{\min\{m, n/2\}}{2}\right\rfloor \quad and \quad e_2=q+1-e_1,$$
	then
	$$h_F(n,q)=(1\pm c)\frac{n}{2}\left[\binom{e_1}{2}+\binom{e_2}{2}+m\binom{d+1}{2}\frac{n}{2}+(n-m)\binom{d}{2}\frac{n}{2}+4 e_1 e_2\right].$$
\end{theorem}
Note that if $n\to\infty$ and $q/n\to\infty$ in Theorem~\ref{thm:assymptotics}, then $d+1=(1+o(1))d=(1+o(1))2q/n$ and $e_1,e_2=(1+o(1))q/2$, resulting in a simpler formula
$$h_F(n,q)=(1+o(1))\frac{9}{8}q^2n.$$

\subsection*{Proof outline of Theorems~\ref{thm:subgraph} and \ref{thm:assymptotics}}

We start by showing an upper bound on $h_F(n,q)$. This can be achieved by counting the number of bowties in an arbitrary graph on $n$ vertices with $ex(n,F)+q$ edges (Lemma~\ref{lem:numbowties}). Therefore the number of bowties in any extremal graph $H$ is small 
and the Graph Removal Lemma (Theorem~\ref{th:removal}) allows us to conclude that $H$ can be made bowtie-free by removing a small number of edges. In addition, since the chromatic number of the bowtie is 3, the Erd\H os-Simonovits Stability Theorem (Theorem~\ref{th:stability}) 
implies that the vertex set of the obtained bowtie-free graph can be partitioned into two sets, $V_1$ and $V_2$, such that almost $|E(T_2(n))|$ edges are present between the two parts. Therefore the original graph $H$ also has almost all edges between $V_1$ and $V_2$ (Lemma~\ref{lem:bipartite}).

The key step is to establish that every edge between $V_1$ and $V_2$ is present (Proposition~\ref{lem:cbg}). In order to achieve this we need to determine the number of bowties containing a given edge and the number of bowties created by inserting an edge. In particular we need to compare the number of bowties containing an edge in $V_1$ or $V_2$ and the number of bowties created by inserting an edge between $V_1$ and $V_2$. 
When the number of edges spanned by $V_1$ and $V_2$ is small, the number of bowties has to be counted very precisely in order to determine which is smaller (Lemma~\ref{lem:cbgsparse}). On the other hand, when the number of edges spanned by $V_1$ and $V_2$ is large, a simpler and less accurate estimate suffices (Lemma~\ref{lem:completebip}).
In both cases there is an edge in $V_1$ or $V_2$ which is contained in more bowties than the number of bowties the insertion of any edge between $V_1$ and $V_2$ would create. Since $H$ is minimal, with respect to the number of copies of bowties, every edge between $V_1$ and $V_2$ must be present.

Let us ignore the triangles spanned by $V_1$ and $V_2$ for the moment, later we will see that no such triangles exist in any extremal graph (Lemma \ref{lem:trifree}). We shall express the number of bowties using an explicit formula. 
A bowtie can be formed in three different ways (see Figure~\ref{fig:types}). Since we ignore triangles spanned by $V_1$ and $V_2$, any triangle in the graph contains exactly one edge spanned by $V_1$ or $V_2$. As a bowtie consists of two triangles, it must have exactly two edges spanned by $V_1$ or $V_2$. We distinguish the following cases, depending on whether both edges are spanned by the same set or not, in addition if they are spanned by the same set we consider the cases when the two edges are adjacent or non-adjacent.

\begin{figure}[h]
	\begin{center}		
		\begin{tikzpicture}[scale=0.9]
		
		\node (X0) at (-5,0) [circle,fill=black, inner sep=1pt] {};
		\node (X1) at (-4.5,0) [circle,fill=black, inner sep=1pt] {};
		\node (X2) at (-4,0) [circle,fill=black, inner sep=1pt] {};
		\node (X3) at (-3.5,0) [circle,fill=black, inner sep=1pt] {};
		\node (X4) at (-4,-2) [circle,fill=black, inner sep=1pt] {};
		
		\draw (X0)--(X1);
		\draw (X2)--(X3);
		\draw (X0)--(X4);
		\draw (X1)--(X4);
		\draw (X2)--(X4);
		\draw (X3)--(X4);
		
		\node (Y0) at (-0.75,0) [circle,fill=black, inner sep=1pt] {};
		\node (Y1) at (-0,0) [circle,fill=black, inner sep=1pt] {};
		\node (Y2) at (0.75,0) [circle,fill=black, inner sep=1pt] {};
		\node (Y3) at (-0.5,-2) [circle,fill=black, inner sep=1pt] {};
		\node (Y4) at (0.5,-2) [circle,fill=black, inner sep=1pt] {};
		
		\draw (Y0)--(Y1);
		\draw (Y1)--(Y2);
		\draw (Y0)--(Y3);
		\draw (Y1)--(Y3);
		\draw (Y1)--(Y4);
		\draw (Y2)--(Y4);
		
		\node (Z0) at (5,0) [circle,fill=black, inner sep=1pt] {};
		\node (Z1) at (4.5,0) [circle,fill=black, inner sep=1pt] {};
		\node (Z2) at (3,0) [circle,fill=black, inner sep=1pt] {};
		\node (Z3) at (3.75,-2) [circle,fill=black, inner sep=1pt] {};
		\node (Z4) at (4.25,-2) [circle,fill=black, inner sep=1pt] {};
		
		\draw (Z0)--(Z1);
		\draw (Z3)--(Z4);
		\draw (Z0)--(Z4);
		\draw (Z1)--(Z4);
		\draw (Z2)--(Z3);
		\draw (Z2)--(Z4);

		\draw (-4,0) ellipse (1.5 and 0.5);	
		\draw (-4,-2) ellipse (1.5 and 0.5);		
		\draw (0,0) ellipse (1.5 and 0.5);	
		\draw (0,-2) ellipse (1.5 and 0.5);	
		\draw (4,0) ellipse (1.5 and 0.5);	
		\draw (4,-2) ellipse (1.5 and 0.5);	
		
		\node at (-4,-3){Type 1};
		\node at (0,-3) {Type 2};
		\node at (4,-3) {Type 3};
		\end{tikzpicture}
		\caption{Types of bowties}\label{fig:types}
	\end{center}
\end{figure}
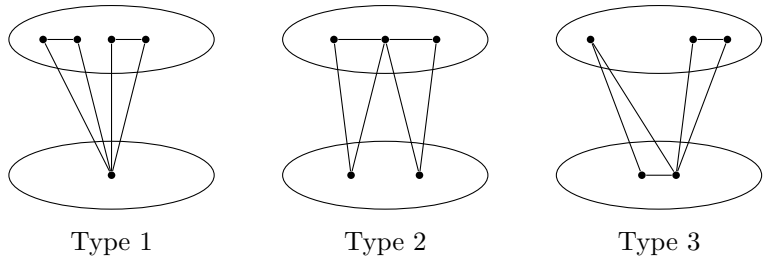

Recall that every edge between $V_1$ and $V_2$ is present. Therefore, any pair of disjoint edges in $V_1$ is contained in $|V_2|$ bowties and similarly any pair of disjoint edges in $V_2$ is contained in $|V_1|$ bowties. If the two edges are adjacent, then they are contained in $|V_2|(|V_2|-1)$ and $|V_1|(|V_1|-1)$ bowties respectively. Finally any two edges, where one edge is spanned by $V_1$ and the other edge is spanned by $V_2$, are contained in $2(n-4)$ bowties.

The previous argument implies that the number of bowties depends only on the size of $V_1$ and $V_2$ and the degree sequence of the graphs spanned by these sets. 
In fact, in the optimal degree sequence the degrees of any two vertices in the same part $V_i$ may differ by at most one. Due to the small number of edges spanned by $V_1$ and $V_2$ one can rearrange the edges inside these parts so that each part spans an almost regular and triangle-free graph (Lemma~\ref{lem:trifree}). Destroying triangles inside a part decreases the number of bowties, justifying our earlier decision to ignore them.

In order to finish the proof of Theorem~\ref{thm:subgraph} we only need to show that $|V_1|=\lceil n/2 \rceil$ or $|V_1|=\lfloor n/2 \rfloor$. Note that if $|V_1|=\lceil n/2 \rceil$ and $|V_2|=\lfloor n/2 \rfloor$, the number of edges spanned by $V_1$ and $V_2$ is exactly $q+1$. If we increase the number of vertices in $V_1$ by $a$ and, in order to keep the total number of vertices unchanged, at the same time we decrease the number of vertices in $V_2$ by the same amount, then the number of edges spanned by $V_1$ and $V_2$ increases by at least $a^2$. In addition, the number of bowties containing edge pairs in $V_1$ decreases (as any such bowtie is either Type 1 or Type 2), while for edge pairs in $V_2$ the number of bowties increases. On the other hand, for a fixed pair of edges, one in $V_1$ and the other in $V_2$, the number of bowties containing both these edges remains unchanged. 
Roughly speaking, the number of bowties created by adding the $a^2$ new edges has to be counterbalanced by the decrease in the number of Type 1 and Type 2 bowties. This would be possible only if there was a large difference between the number of edges in $V_1$ and $V_2$. However, we show that this is not the case (Lemma~\ref{lem:smalldif}) and thus $a$ must be 0.

Bounding the difference between the number of edges spanned by $V_1$ and $V_2$ has an additional advantage. Together with the exact size of $V_1$ and $V_2$ it also implies that, disregarding a couple of vertices, when looking at the graphs spanned by $V_1$ and $V_2$ the difference in the degree of any pair of vertices is at most 1 (Lemma~\ref{lem:nosmalldeg}). In fact, in one of the partitions almost every vertex has the same degree (Lemma~\ref{lem:partitiondegree}). These provide us with a good approximation on the degree sequence of the graphs spanned by $V_1$ and $V_2$ and thus also on the total number of bowties.

\subsection*{Organisation of the paper}
The remainder of the paper is divided into four parts. In Section~\ref{sec:cbg} we show that any extremal graph $H$ contains a complete spanning bipartite subgraph (Proposition~\ref{lem:cbg}). Section~\ref{sec:opt} states any additional lemmas needed for Theorems~\ref{thm:subgraph} and \ref{thm:assymptotics} and gives the proofs of these results. The proof of the remaining technical lemmas can be found in Section~\ref{sec:proofs}. Section~\ref{conclusion} contains some concluding remarks.

\section{Complete Bipartite Subgraph: proof of Proposition~\ref{lem:cbg}}\label{sec:cbg}

Throughout this section let $H$ be a graph on $n$ vertices and $\ex(n,F)+q$ edges containing the minimal number of bowties.
Let $V:=V(H)$ denote the vertex set of $H$ and $E:=E(H)$ its edge set. 
The following proposition is vital in proving Theorem~\ref{thm:subgraph}.

\begin{proposition}\label{lem:cbg} 
	For an arbitrary graph $H$ on $n$ vertices and $\ex(n,F)+q$ edges containing $h_F(n,q)$ bowties,
	where $q=o(n^2)$ and $n$ is large enough, admits a partition $V=V_1\cup V_2$ such that $E(K(V_1,V_2))\subseteq E$, $|V_1|,|V_2|=(1+o(1))n/2$ and the number of edges spanned by each of $V_1$ and $V_2$ is at most $4q+4$.
\end{proposition}

The remainder of this section is devoted to proving Proposition~\ref{lem:cbg}. 
Let $\{V_1,V_2\}$ be a max-cut of $H$. The crucial part of the proof is to show that both parts have size $(1+o(1))n/2$ and every edge between the parts is present, i.e.\ $E(K(V_1,V_2))\subseteq E$.

The key tools used for the proof are the Graph Removal Lemma (Theorem~\ref{th:removal}) and the Stability Theorem (Theorem \ref{th:stability}). Roughly speaking, the graph removal lemma states that if a graph contains only a few copies of bowties, which holds for $H$ (Lemma~\ref{lem:numbowties}), then removing a few edges makes it bowtie-free. Together with the Stability Theorem this implies that $V_1$ and $V_2$ both have size roughly $n/2$ and most edges between the parts are present (Lemma~\ref{lem:bipartite}).

For $i=1,2$ let $\partition{i}=|V_i|$. Denote the edges spanned by $V_1$ or $V_2$ by $B$ and set $\badedge=|B|$. We call the edges in $B$ \emph{bad}. The degree sequence of several graphs on $V$ play a crucial role in counting the number of bowties. For any $v\in V$ and $E'\subseteq \binom{V}{2}$ let $d_{E'}(v)$ be the degree of vertex $v$ in $(V,E')$ (the graph with vertex set $V$ and edge set $E'$). In the special case when $E'=B$ we call $d_B(v)$ the \emph{bad degree}.

Our proof strategy is to show that there exists a bad edge which is contained in more bowties then inserting an edge between the parts would create, should such an edge be missing.
Therefore removing the bad edge destroys more bowties than inserting the edge between the parts creates leading to a graph with fewer bowties and implying that every edge between $V_1$ and $V_2$ is present. 
Showing this requires a very precise analysis, when the number of bad edges is small (Lemma~\ref{lem:cbgsparse}). However this fails once the number of bad edges becomes large and an alternate proof is required (Lemma~\ref{lem:completebip}).

In order to prove the above lemmas we need a lower bound on the maximal number of bowties a bad edge is contained in (Lemma~\ref{lem:lowerbowtie}) and an upper bound on the number of bowties created when an edge between the parts is inserted (Lemma~\ref{lem:missing}). The latter is more difficult. Roughly speaking an upper bound on the bad degree leads to an upper bound on the number of triangles containing the inserted edge, which leads to an upper bound on the number of bowties.
At first we can only prove a weak upper bound on the bad degree (Lemma~\ref{lem:bounddegree}). However if the vertex has many neighbours in the other part, which is true for the majority of the vertices, a significantly better bound exists (Lemma~\ref{lem:newupperbad}). In fact, once we show that every vertex has many neighbours in the other part (Lemma~\ref{lem:nobadvertices}) the tighter bound applies to every vertex.

As mentioned earlier we need to show that the number of bowties in $H$, i.e.\ $h_F(n,q)$, is small. 

\begin{lemma}\label{lem:numbowties}
	For 
	$q\leq n^2/20$ we have that 
	$$h_F(n,q)\leq (q+1)^2\left(13n/4+13\right).$$
\end{lemma}

\begin{proof}
	In order to prove the statement we construct a graph $G$ satisfying $|V(G)|=n$, $|E(G)|=\ex(n,F)+q$ and $\# F(G)\leq (q+1)^2(13n/4+13)$.
	Partition the vertex set of $G$ into two parts $U_1$ and $U_2$ such that $|U_1|=\lceil n/2 \rceil$ and $|U_2|= \lfloor n/2 \rfloor$. Every edge between $U_1$ and $U_2$ will be included in the graph, i.e.\ $K(U_1,U_2)\subseteq G$. This determines $\lceil n/2\rceil \lfloor n/2\rfloor$ edges in $G$ and thus we only need to establish the position of the remaining $q+1$ edges.
	
	Denote by $n'=\lfloor n/4 \rfloor$. Let $W_{1},W_{2}$ be disjoint subsets of $U_1$ of size $n'$. The remaining $q+1$ edges are placed between $W_{1}$ and $W_{2}$ such that the degrees are as equal as possible. In particular, for $i=1,2$, $(q+1)\mod{n'}$ vertices in $W_{i}$ have $\lceil (q+1)/n' \rceil$ neighbours in $W_{3-i}$, while the remaining vertices have $\lfloor (q+1)/n' \rfloor$ neighbours in $W_{3-i}$. Note that such a construction is possible by Lemma~\ref{lem:trifreeeven}, as $\lfloor(q+1)/n'\rfloor <n'$ for sufficiently large $n$. 
	
	In order to determine the number of bowties we count for every vertex $v$ the number of bowties in $G$ where the central vertex is $v$.
	For $v\in V(G)$ any two edge disjoint triangles containing $v$ forms a bowtie. Since no edges are spanned by $U_2$, every triangle must contain an edge in $U_1$ and any such edges must be between $W_1$ and $W_2$. Therefore for $v\in U_2$ there are at most $\binom{q+1}{2}$ bowties centred at $v$. 
	
	Now consider bowties centred at $v\in W_i$ for $i=1,2$. Recall that any triangle must contain an edge in $U_1$ and a vertex in $U_2$. Therefore any bowtie centred at $v$ must contain two different vertices in $N(v)\cap U_{1}$ and two vertices in $U_2$. These can be selected in at most
	$$\binom{|N(v)\cap U_1|}{2}\frac{n^2}{4}$$
	ways, giving an upper bound on the number of bowties.	
	Let $k\equiv q+1 \pmod{n'}$. Then we have
	\begin{align*}
	\#F(G)&\leq 2k \binom{\lceil (q+1)/n' \rceil}{2}\frac{n^2}{4}+2(n'-k)\binom{\lfloor (q+1)/n' \rfloor}{2}\frac{n^2}{4}+\frac{n}{2}\binom{q+1}{2}\\
	&\leq\frac{n}{2}\left[k \binom{\lceil (q+1)/n' \rceil}{2}n+(n'-k)\binom{\lfloor (q+1)/n' \rfloor}{2}n+\frac{(q+1)^2}{2}\right]\\
	&\leq\frac{n}{4}\left[n'\left(\frac{q+1}{n'}\right)^2n+kn\frac{2(q+1)}{n'}+(q+1)^2\right].
	\end{align*}
	Note that $k\leq q+1$ and $n'\geq n/4-1$, therefore
	\begin{align*}
	\#F(G)&\leq \frac{n}{4}(q+1)^2\left[\frac{3n}{n/4-1}+1\right]
	\leq (q+1)^2\left(13n/4+13\right)
	\end{align*}
	for large enough $n$.
\end{proof}

Next we state the Graph Removal Lemma and the Stability Theorem. Using them we show that the vertex set of the extremal graph $H$ can be partitioned into two sets $V_1,V_2$ such that both sets contain roughly $n/2$ vertices and most edges between $V_1$ and $V_2$ are present. 

\begin{theorem}[Graph Removal Lemma, see e.g.\ \cite{komlos+simonovits:96}]\label{th:removal}
	Let $\Gamma$ be a graph with $f$ vertices. Then for every $\varepsilon_1>0$ there exists $\varepsilon_2>0$ such that every graph $H$ with $n\ge 1/\varepsilon_2$ vertices and at most $\varepsilon_2 n^f$ copies of $\Gamma$ can be made $\Gamma$-free by removing at most $\varepsilon_1 n^2$ edges.
\end{theorem}

\begin{theorem}[Stability Theorem \cite{erdos:67a,simonovits:68}]
	\label{th:stability}
	Let $r \geq 2$ and $\Gamma$ be a graph with chromatic number $r+1$. Then for every $\varepsilon_1>0$ there exists $\varepsilon_2>0$ such that every $\Gamma$-free graph on $n\ge 1/\varepsilon_2$ vertices with at least $|E(T_{r}(n))|-\varepsilon_2 n^2$ edges contains an $r$-partite subgraph with at least $|E(T_{r}(n))|-\varepsilon_1 n^2$ edges.
\end{theorem}

\begin{lemma}\label{lem:bipartite}
	Let $\{V_1,V_2\}$ be a max-cut of $H$. For every $\delta_1>0$ there exist $\delta_2>0$ and $n_0$ such that for every $n>n_0$ and $q+1< \delta_2 n^2$ we have 
	$$|E(H)\triangle E(K(V_1,V_2))|\leq \delta_1 n^2 +q +1$$ 
	and 
	$$n/2-2\sqrt{\delta_1}n \leq |V_1|,|V_2|\leq n/2+2 \sqrt{\delta_1} n.$$
\end{lemma}

\begin{proof}
	First we show that the number of edges between $V_1$ and $V_2$ is at least $\lceil n/2 \rceil \lfloor n/2 \rfloor -\delta_1 n^2$. Since $\{V_1,V_2\}$ is a max-cut, this follows if $H$ contains a bipartite graph with at least $\lceil n/2 \rceil \lfloor n/2 \rfloor -\delta_1 n^2$ edges. Note that the chromatic number of the bowtie is 3. Therefore, Theorem~\ref{th:stability} implies that there exists a constant $\varepsilon_1>0$ such that every bowtie-free graph with at least $\lceil n/2 \rceil \lfloor n/2 \rfloor -\varepsilon_1 n^2$ edges contains a bipartite subgraph with at least $\lceil n/2 \rceil \lfloor n/2 \rfloor -\delta_1 n^2$ edges.
	
	Therefore, $H$ contains such a bipartite graph if it has a bowtie-free subgraph with at least $\lceil n/2 \rceil \lfloor n/2 \rfloor -\varepsilon_1 n^2$ edges. Theorem~\ref{th:removal} implies that there exists an $\varepsilon_2>0$ such that every graph with at most $\varepsilon_2 n^5$ copies of bowties can be made bowtie-free by removing at most $\varepsilon_1 n^2$ edges. This holds for every $q$ satisfying $q+1<\sqrt{\varepsilon_2}n^2/2$ because by Lemma~\ref{lem:numbowties} the number of bowties in $H$ is at most
	$$(q+1)^2\left(13n/4+13\right) \leq \varepsilon_2 n^5.$$
	In fact, after the removal of these edges the graph still contains at least 
	$$\lceil n/2 \rceil \lfloor n/2 \rfloor +q+1 -\varepsilon_1 n^2 \geq \lceil n/2 \rceil \lfloor n/2 \rfloor -\varepsilon_1 n^2 $$
	edges.
	
	Thus we have a bipartite subgraph with at least $\lceil n/2 \rceil \lfloor n/2 \rfloor -\delta_1 n^2$ edges and the first statement follows. Note that a complete bipartite graph on $n$ vertices where one part has size at least $n/2+2\sqrt{\delta_1}n$ contains at most 
	$$\frac{n^2}{4}-4\delta_1 n^2\stackrel{\delta_1 n^2>1/8}{\leq} \lceil n/2 \rceil \lfloor n/2 \rfloor-2\delta_1 n^2$$
	edges, implying the required bounds on the size of the individual parts.
\end{proof}

Let $\eps>0$ be sufficiently small. 
Apply Lemma~\ref{lem:bipartite} for $\delta_1=\eps^2/4$ which gives some $\delta_2>0$.
Since $q=o(n^2)$, we have that $q+1\leq \delta_2 n^2$, for large enough $n$. 
Thus $\{V_1,V_2\}$, a max-cut of $H$, satisfies 
$$n/2-\eps n \leq |V_1|,|V_2|\leq n/2+ \eps n$$ 
and 
$$|E(H)\triangle E(K(V_1,V_2))|\leq \eps^2 n^2/4 +q +1\leq \eps^2 n^2.$$
Note that 
\begin{equation}\label{eq:boundonB}
\badedge\leq |E(H)\triangle E(K(V_1,V_2))|\leq \eps^2 n^2.
\end{equation} 

In order to establish bounds on $d_B(v)$ we need to determine the number of bowties any edge is contained in.

\begin{lemma}\label{lem:upperdestroyed}
	Every edge in $H$ is contained in less than $13\badedge n$ bowties. 
\end{lemma}

\begin{proof}
	Assume for contradiction that there exists an edge $f$ contained in at least $13\badedge n$ bowties. We will show that there exists a pair of non-adjacent vertices $w_1,w_2\in V$ such that inserting the edge $\{w_1,w_2\}$ creates less than $13\badedge n$ bowties. 
	Thus the graph created by removing $f$ and inserting $\{w_1,w_2\}$ has less bowties and contradicts the minimality of $H$.

	Note that the number of vertices $w\in V_1$ with $d_{B}(w)\geq 6\badedge /n$ is at most $2\badedge /(6\badedge /n)=n/3.$ This implies that the number of vertices in $V_1$ with $d_{B}(w)< 6\badedge /n$ is at least $(1-2\eps)\frac{n}{2}-\frac{n}{3}\geq n/8$.
	Since for each of these vertices $d_{B}(w)<6\badedge /n\leq 6 \eps^2 n$, two of them must be non-adjacent. 
	
	Let $w_1,w_2\in V_1$ be a pair of non-adjacent vertices with $d_{B}(w_1),d_{B}(w_2)\leq 6\badedge /n$ and we count the bowties which would be created if we would insert this edge. 
	Note that any such bowtie must contain at least one more edge in $B$. We start with the bowties which contain an edge in $B$ adjacent to $\{w_1,w_2\}$. In order to count the number of these bowties we first count the number of triangles containing an edge in $B$ adjacent to $w_1$ or $w_2$. Note that for $w_1$ there are $d_B(w_1)$ such edges and each of these edges creates a triangle with at most $\partition{2}+d_B(w_1)$ vertices. Therefore there are at most $((1+2\eps)n/2+6\badedge /n)6 \badedge/n$ such triangles, and similarly for $w_2$. Thus the number of triangles containing an edge in $B$ adjacent to $w_1$ or $w_2$ is at most $((1+2\eps)n/2+6\badedge /n)12\badedge /n$.
	Note that the codegree of $w_1,w_2$ is at most $n$. Consequently the total number of bowties created in this fashion is at most
	$$\left((1+2\eps)\frac{n}{2}+\frac{6\badedge }{n}\right)\frac{12\badedge }{n}n\stackrel{\eqref{eq:boundonB}}{<} 8\badedge n.$$
	
	Now consider the case when the bowtie contains an edge in $B$ which is not adjacent to $\{w_1,w_2\}$. 
	Select an edge $e\in B$, this can be done in $\badedge$ ways. The final vertex of the bowtie can be picked in at most $n$ ways. Note that once the central vertex of the bowtie has been selected the structure of the bowtie has been determined, as it has to contain the edges $e$ and $\{w_1,w_2\}$ and these two edges have to be disjoint. Therefore we create at most $5\badedge n$ bowties this way.
	
	In total we created less than $13\badedge n$ bowties, leading to a contradiction.
\end{proof}

Using Lemma~\ref{lem:upperdestroyed} we obtain an upper bound on $d_B(v)$.

\begin{lemma}\label{lem:bounddegree}
	For every $v\in V$ we have $d_B(v)< 5\eps^{2/3} n$. 
\end{lemma}

\begin{proof}
	Assume for contradiction that there exists a vertex $v\in V$ such that $d_B(v)\geq 5\eps^{2/3} n$. Without loss of generality assume that $v\in V_1$. Because $\{V_1,V_2\}$ is a max-cut of $H$ we have that $|N(v)\cap V_2|\geq  5\eps^{2/3} n$. Let $N_1\subseteq N(v)\cap V_1$ and $N_2\subseteq N(v)\cap V_2$ be sets of size exactly $5\eps^{2/3}n$.
	Since at most $\eps^2 n^2$ edges are missing between $V_1$ and $V_2$, there can be only $\eps^2n^2$ edges missing between $N_1$ and $N_2$. Therefore, there are at least $(25\eps^{4/3}-\eps^2) n^2\geq 24\eps^{4/3}n^2$ edges between $N_1$ and $N_2$ and each of these edges is adjacent to at most $10\eps^{2/3}n$ other edges in this set. 
	In addition, by the pigeonhole principle there must exist a vertex $u\in N_1$ such that $|N(u)\cap N_2|\geq 2\eps^{2/3} n$.

	Next we count the number of bowties containing the edge $\{u,v\}$. 
	Note that any pair of disjoint edges between $N_1$ and $N_2$ creates a bowtie, where the central vertex is $v$. Recall that $u$ is contained in at least $2\eps^{2/3}n$ edges, where the other end is in $N_2$. In addition, each of these edges is disjoint from at least $24\eps^{4/3}n^2-10\eps^{2/3} n$ edges between $N_1$ and $N_2$.
	Therefore, $\{u,v\}$ is contained in at least 
	$$2\eps^{2/3} n (24\eps^{4/3}n^2-10\eps^{2/3} n)>25\eps^2n^3$$
	bowties, when $n$ is large enough.
	
	On the other hand, Lemma~\ref{lem:upperdestroyed} implies that any edge can be contained in at most $13\badedge n$ bowties and by the  bound \eqref{eq:boundonB} giving $\badedge\leq \eps^2n^2$, we have $13\badedge n\leq 13 \eps^2n^3$, leading to a contradiction.
\end{proof}

However, for vertices which have many neighbours in the opposite part, which holds for most vertices, we establish a tighter bound on $d_B(v)$.
Let 
$$M=E(K(V_1,V_2))\setminus E(H).$$ 
We call the edges in $M$ \emph{missing}. Recall that $d_M(v)$ is the degree of vertex $v$ in the graph $(V,M)$. Let $S_i\subseteq V_i$ be the set of vertices $v\in V_i$ satisfying $d_M(v)\geq \eps^{1/4}n$. 
Since $q+1+|M|\leq \badedge$, we have
\begin{equation}\label{eq:numbadvertices}
|S_i|\leq \frac{|M|}{\eps^{1/4}n}\leq \frac{\badedge}{\eps^{1/4}n}.
\end{equation}
Next we will consider the properties of the graph spanned by $V'=V\backslash (S_1\cup S_2)$, denoted by $H'$. 
Let $V_1'=V_1\setminus S_1$ and $V_2'=V_2\setminus S_2$. The set of edges spanned by $V_1'$ and $V_2'$ are denoted by $B_1'$ and $B_2'$ respectively. Also let $B'=B_1'\cup B_2'$.
Note that for $i=1,2$, for any vertex $v\in V_i'$ we have $d_B(v)\leq d_{B'}(v)+|S_i|$. Since $|S_i|$ is bounded from above by $b/(\eps^{1/4}n)$, any upper bound on $d_{B'}(v)$ also provides an upper bound on $d_B(v)$.

\begin{lemma}\label{lem:newupperbad}
	For every $v\in V'$ we have $d_{B'}(v)\leq 80\badedge/n+1$. 
\end{lemma}

\begin{proof}
	Assume for contradiction that there exists $v\in V'$ such that $d_{B'}(v)>80\badedge/n+1$. Without loss of generality assume $v\in V_1'$. We will show that in this graph there is an edge which is contained in at least $13\badedge n$ bowties which together with Lemma~\ref{lem:upperdestroyed} leads to a contradiction.
	
	Note that any vertex in $V_1'$ has at least $\partition{2}-\eps^{1/4}n$ neighbours in $V_2$ and thus any pair of vertices in $V_1'$ has at least $\partition{2}-2\eps^{1/4}n\geq (1-2\eps)n/2-2\eps^{1/4}n$ common neighbours. 
	Therefore every edge in $B_1'$ is contained in at least $(1-2\eps)n/2-2\eps^{1/4}n$ triangles. 
	Any pair of triangles each containing a different edge in $B_1'$ and a vertex in $V_2$ is edge disjoint if the vertex in $V_2$ differs. Therefore, any edge in $B'$ adjacent to $v$ is contained in at least
	$$(d_{B'}(v)-1)\left((1-2\eps-4\eps^{1/4})\frac{n}{2}-1\right)^2\geq 16\badedge n$$
	bowties, a contradiction.
\end{proof}

Recall that $q+1+|M|\leq \badedge$. When $\badedge$ is small, we have $S_1,S_2=\emptyset$, implying $H=H'$ and $d_{B'}(v)=d_B(v)$. 
Therefore, the upper bound on the degree in the graph $(V,B')$ in Lemma~\ref{lem:newupperbad} is actually an upper bound on the degree in $(V,B)$.
This enables us to show that $K(V_1,V_2)$ is a subgraph of $H$.

\begin{lemma}\label{lem:cbgsparse}
	When $\badedge< \eps^{1/4}n$ we have that $M=\emptyset$.
\end{lemma}

\begin{proof}
	Assume for contradiction that $M\neq\emptyset$. 
	Since $\badedge < \eps^{1/4}n$, by \eqref{eq:numbadvertices} we have $|S_1|,|S_2|<1$, implying that $S_1=S_2=\emptyset$. Together with Lemma~\ref{lem:newupperbad} we have $d_{B}(w)\leq 80\badedge/n+1< 80\eps^{1/4}+1<2$ for every $w\in V$. Therefore, $d_B(w)$ is at most one, which in turn implies that the edges in $B$ are disjoint. 
	
	Without loss of generality assume $\badedgep{1}\geq \badedgep{2}$. For every edge $\{w_1,w_2\}\in B_1$ let $t_{\{w_1,w_2\}}$ denote the number of triangles containing either $w_1$ or $w_2$ and an edge in $B_2$. 
	
	We show that if $B_2\neq \emptyset$, then there exists an edge $e\in B_1$ with $t_e\geq \badedgep{2}+1$. Assume for contradiction that there is no such edge. 
	For any $\{w_1,w_2\}\in B_1$ consider the $2\badedgep{2}$ potential triangles which could contribute to $t_{\{w_1,w_2\}}$. Since the edges in $B_2$ are disjoint, these triangles are edge disjoint except for the edge in $B_2$. In order for $t_{\{w_1,w_2\}}\leq \badedgep{2}$ we must have that $\badedgep{2}$ of these triangles are missing at least one edge and thus $d_M(w_1)+d_{M}(w_2)\geq \badedgep{2}$. Since the edges in $B_1$ are also disjoint, we have
	$$|M|\geq \sum_{\{w_1,w_2\}\in B_1}\left(d_M(w_1)+d_{M}(w_2)\right)\geq \badedgep{1}\badedgep{2}.$$
	Note that because $|M|\leq \badedge -q-1$ and $q\geq 1$, we have $|M|\leq \badedge -2=\badedgep{1}+\badedgep{2}-2$. However, no positive integer pair $\badedgep{1}\geq\badedgep{2}\geq 1$ satisfies $\badedgep{1}+\badedgep{2}-2\geq \badedgep{1}\badedgep{2}$, resulting in a contradiction.
	
	Select $\{w_1,w_2\}\in B_1$ such that $t_{\{w_1,w_2\}}$ is maximal. We will give a lower bound on the number of bowties destroyed when $\{w_1,w_2\}$ is removed. First we consider bowties which have an edge in $B_2$. 
	Note that a triangle $w_1,x,y$, where $x,y\in V_2$, forms a bowtie with any triangle containing both $w_1$ and $w_2$, but not $x$ and $y$. Also the codegree of $w_1,w_2$, not including $x$ and $y$, is at least $(1-2\eps)n/2-|M|-2\geq (1-2\eps)n/2-\badedge $. Since $\badedgep{2}+\mathbbm{1}_{B_2\neq\emptyset}=0$, when $b_2=0$ we have $t_{\{w_1,w_2\}}\geq \badedgep{2}+\mathbbm{1}_{B_2\neq\emptyset}$. Thus removing $\{w_1,w_2\}$ destroys at least
	$$(\badedgep{2}+\mathbbm{1}_{B_2\neq\emptyset})\left((1-2\eps)\frac{n}{2}-\badedge \right)$$
	such bowties.
	Now we consider the case when the bowtie contains a second edge in $B_1$. Note that any four vertices in $V_1$ have at least $(1-2\eps)n/2-\badedge $ common neighbours. Recall that the edges in $B$ are disjoint. Thus any pair of edges in $B_1$ are contained in at least $(1-2\eps)n/2-\badedge $ bowties. Therefore, at least
	$$(\badedgep{1}-1)\left((1-2\eps)\frac{n}{2}-\badedge \right)$$
	additional bowties are destroyed when $\{w_1,w_2\}$ is removed. In total, this leads to at least
	$$(\badedge -1+\mathbbm{1}_{B_2\neq\emptyset})\left((1-2\eps)\frac{n}{2}-\badedge \right)\geq (\badedge -1+\mathbbm{1}_{B_2\neq\emptyset})\left((1-2\eps)\frac{n}{2}-\eps^{1/4}n\right)$$
	destroyed bowties.
	
	Since $M\neq \emptyset$, there exists $\{u,v\}\in M$. Without loss of generality assume $u\in V_1$ and $v\in V_2$. Now we analyse the number of bowties created when $\{u,v\}$ is added to the graph. 
	
	We first consider the case when $B_2=\emptyset$. Since $B_2=\emptyset$, $v$ has no neighbours in $B_2$ and thus inserting $\{u,v\}$ can create at most one triangle, with the third vertex being $z\in V_1$, if it exists. The number of bowties created is equal to the number of triangles containing exactly one of $u,v,z$ as these form a bowtie with the triangle $u,v,z$. The only neighbour of $u$ in $V_1$ is $z$ and vice versa. This means that the only way these vertices can be in a triangle is if the triangle contains an edge in $B_2$. However as $B_2=\emptyset$ no such triangles exist. Similarly, the only way $v$ can be in a triangle is if the triangle contains an edge in $B_1$, so there are at most $\badedgep{1}= \badedge $ such triangles. Thus removing $\{w_1,w_2\}$ destroys at least
	$$(\badedge -1)\left((1-2\eps)\frac{n}{2}-\eps^{1/4}n\right)$$
	bowties, while inserting $\{u,v\}$ creates at most $\badedge $. Since $\badedge \geq |M|+q+1\geq 3$ we destroy more triangles than we created, which leads to a contradiction.
	
	On the other hand, if $B_2\neq\emptyset$, then $u$ can have a neighbour in $V_1$ and $v$ can have a neighbour in~$V_2$. Denote these vertices by $z_1$ and $z_2$ respectively, if it exists. Similarly as before, we are interested in the number of triangles containing exactly one of $u,v,z_1$ or one of $u,v,z_2$. We start with the $u,v,z_1$ triangle. Similarly as before, for each of $u$ and $z_1$ there are at most $\badedgep{2}$ such triangles where the additional vertices of the triangle are in $V_2$, and for $v$ there are at most $\badedgep{1}$ such triangles where the additional vertices are in $V_1$. Now the only neighbour of $u$ in $V_1$ is $z_1$ and vice versa, so we have counted every triangle containing $u$ or $z_1$. However, $v$ has a neighbour $z_2\in V_2$ and every vertex in $V_1\backslash \{u,z_1\}$ can form a triangle with $\{v,z_2\}$. Therefore, the total number of bowties containing the $\{u,v,z_1\}$ triangle is at most $\badedgep{1}+2\badedgep{2}+\partition{1}$. Analogously, one can show for the $\{u,v,z_2\}$ triangle that the number of bowties it is contained in is at most $2\badedgep{1}+\badedgep{2}+\partition{2}$.
	So in total at most $3\badedge +n$ bowties have been created, by adding the edge $\{u,v\}$. Now this has to be more than the number of triangles destroyed, so
	$$3\badedge +n\geq \badedge \left((1-2\eps)\frac{n}{2}-\eps^{1/4}n\right),$$
	which leads to a contradiction as $\badedge \geq 3$.
\end{proof}

In the remainder of this section we prove that $M=\emptyset$ when $\badedge$ is large. 
We start by showing that $S_1=S_2=\emptyset$ holds in this case as well. Namely, should there exist a vertex $v\in S_1\cup S_2$ we show that removing many edges from $B$ and inserting them between $V_1$ and $V_2$ such that they are adjacent to $v$ decreases the number of bowties in the graph. As a first step we need to find a suitable set of edges, which when removed from the graph destroy many copies of bowties.

\begin{lemma}\label{lem:lowerbowtie}
	For every integer $k\leq |B'|/2$ there exists $D\subseteq B'$ with $|D|=k$ such that removing all edges in $D$ destroys at least
	$$k\frac{\badedge n}{8}-\frac{n k^2}{2}$$
	bowties.
\end{lemma}

\begin{proof}
	Note that
	\begin{align}
	\sum_{v\in S_1}d_{B}(v)+\sum_{v\in S_2}d_{B}(v)
	&\stackrel{\mathrm{Lem.}~\ref{lem:bounddegree}}{\leq}(|S_1|+|S_2|)5\eps^{2/3}n\nonumber \\
	&=10\eps^{5/12}(|S_1|+|S_2|)\eps^{1/4}n/2\nonumber\\
	&\stackrel{\eqref{eq:numbadvertices}}{\leq} 10\eps^{5/12}\badedge.\label{eq:lostdegrees}
	\end{align}
	Thus
	\begin{equation}\label{eq:bprime}
	|B'|\geq \badedge -\sum_{v\in S_1}d_{B}(v)-\sum_{v\in S_2}d_{B}(v)\stackrel{\eqref{eq:lostdegrees}}{\geq} (1-10\eps^{5/12})\badedge >0.
	\end{equation}

	Therefore, $B'$ is non-empty. Without loss of generality we assume that $V_1$ spans at least as many edges in $B'$ as $V_2$.
	Note that any four vertices in $V_1'$ have at least $(1-2\eps)n/2-4\eps^{1/4}n$ common neighbours. Select an arbitrary pair of edges $\{w_1,w_2\},\{w_3,w_4\}\in B_1'$. If $\{w_3,w_4\}$ is disjoint from $\{w_1,w_2\}$, then for every shared neighbour of $w_1,w_2,w_3,w_4$ we have a bowtie. On the other hand, if $\{w_3,w_4\}$ intersects $\{w_1,w_2\}$, then every pair of shared neighbours results in a bowtie. Thus the number of bowties containing $\{w_1,w_2\}$ and $\{w_3,w_4\}$ is at least
	$$(1-2\eps-8\eps^{1/4})\frac{n}{2}.$$
	Therefore, removing arbitrary $k$ edges from $B_1'$ destroys at least
	\begin{align*}
	(1-2\eps-8\eps^{1/4})\frac{n}{2}\left(k\left(\frac{|B'|}{2}-k\right)\right)
	&\stackrel{\eqref{eq:bprime}}{\geq} (1-2\eps-8\eps^{1/4})\frac{n}{2}\left(k\left((1-10\eps^{5/12})\frac{\badedge }{2}-k\right)\right)\\
	&\geq k\frac{\badedge n}{8}-\frac{n k^2}{2}
	\end{align*}
	bowties.
\end{proof}

Next we analyse the number of bowties created if we insert edges between $V_1$ and $V_2$.

\begin{lemma}\label{lem:missing}
	Assume $\{u,v\}\in M$ with $u\in V_1$ and $v\in V_2$. Let $G$ be a graph created from $H$ by removing an arbitrary set of edges $D\subseteq E(H)$ and inserting an arbitrary set of edges from $M\setminus\{u,v\}$. 
	If 
	\begin{equation}\label{eq:boundedbaddegree}
	|N_G(u)\cap V_1'|,|N_G(v)\cap V_2'|\leq 80\badedge /n+1,
	\end{equation}
	then inserting $\{u,v\}$ into the graph $G$ creates at most $4\eps^{1/4}\badedge n+6n$ bowties. 
\end{lemma}

\begin{proof}
	Note that adding the edges in $M\setminus \{u,v\}$ leaves $B$ unchanged and removing the edges in $D$ can only decrease the size of $B$.
	
	We start by determining an upper bound on the number of triangles in $G$ a given vertex is contained in. Consider an arbitrary vertex $w\in V$. Without loss of generality let $w\in V_1$ and we consider the triangles depending on the number of neighbours it contains in $V_2$. Should the additional vertices of the triangle both be in $V_1$ or $V_2$, they form an edge in $B$ and thus there are at most $\badedge $ such triangles. On the other hand, if the triangle contains two vertices in $V_1$ and one vertex in $V_2$, then the second vertex in $V_1$ has to be a neighbour of $w$, which can be chosen in $d_{B\setminus D}(w)$ ways and the vertex in $V_2$ can be chosen in at most $n$ ways. Therefore, the number of triangles containing $v$ is at most $\badedge +d_{B\setminus D}(v)n$. This together with Lemma~\ref{lem:bounddegree} implies that the number of triangles containing any vertex can be bounded from above by
	\begin{equation}\label{eq:trigen}
	\badedge +5\eps^{2/3}n^2.
	\end{equation}
	
	However, for $w\in V'\cup\{u,v\}$ a stronger bound holds. 
	Note that for a vertex $w\in V_1$ the value of $d_{B\setminus D}(w)$ is determined by the number of neighbours $w$ has in $V_1'$ and in $S_1$. For $w\in V_1'\cup \{u\}$, Lemma~\ref{lem:newupperbad} and \eqref{eq:boundedbaddegree} imply that the first of these two terms can be bounded from above by $80\badedge /n+1$, while the second term is at most $\eps^{-1/4}\badedge /n$ by \eqref{eq:numbadvertices}. An analogous argument holds for $w\in V_2'\cup\{v\}$.
	Thus the number of triangles that $w\in V'\cup\{u,v\}$ is contained in is at most
	\begin{equation}\label{eq:trispec}
	\badedge +\left(\frac{80\badedge }{n}+1+\frac{\badedge }{\eps^{1/4}n}\right)n=81\badedge +n+\eps^{-1/4}\badedge \leq 3\eps^{-1/4}\badedge +n.
	\end{equation}

	Inserting the edge $\{u,v\}$ creates two types of triangles, depending on whether the third vertex is in $V'$ or not. We first consider the case when the third vertex is in $V'$. Note that due to our conditions there are at most $|N_G(u)\cap V_1'|+|N_G(v)\cap V_2'|\leq 2(80\badedge /n+1)$ such triangles.
	Due to \eqref{eq:trispec} every vertex of these triangles is contained in at most $3\eps^{-1/4}\badedge +n$ triangles, thus the number of bowties is bounded from above by
	\begin{align}
	2\left(\frac{80\badedge }{n}+1\right)3(3\eps^{-1/4}\badedge +n)&=1440\eps^{-1/4}\frac{\badedge ^2}{n}+18\eps^{-1/4}\badedge +480\badedge +6n\nonumber\\
	&\stackrel{\eqref{eq:boundonB}}{\leq} \eps\badedge n+18\eps^{-1/4}\badedge +6n \nonumber\\
	&\leq2\eps\badedge n+6n\label{eq:boundtriangle},
	\end{align}
	where the last inequality holds, since $n$ is large enough.
	
	In addition, inserting $\{u,v\}$ also creates at most $|S_1|+|S_2|$ triangles with the third vertex in $S_1\cup S_2$. Thus, by \eqref{eq:numbadvertices} and \eqref{eq:trigen}
	an upper bound on these bowties is
	$$2\frac{\badedge }{\eps^{1/4}n}3\left(\badedge +5\eps^{2/3}n^2\right)=\frac{6\badedge ^2}{\eps^{1/4}n}+30\eps^{5/12}\badedge n\stackrel{\eqref{eq:boundonB}}{\leq}2\eps^{1/4}\badedge n.$$
	This together with \eqref{eq:boundtriangle} implies that at most
	$$4\eps^{1/4}\badedge n+6n$$
	bowties have been created.
\end{proof}

Now we can show that both $S_1$ and $S_2$ is empty and thus $H=H'$.

\begin{lemma}\label{lem:nobadvertices}
	We have that $S_1,S_2=\emptyset$.
\end{lemma}

\begin{proof}
	Assume for contradiction that $S_1$ or $S_2$ is non-empty. Without loss of generality assume $v\in S_1$. Then we have that $d_M(v)\geq \eps^{1/4}n$.
	In addition, by \eqref{eq:boundonB} and \eqref{eq:numbadvertices} we have $|S_2|\leq \badedge \eps^{-1/4}n^{-1}\leq \eps^{7/4}n$ and thus there exists $U\subseteq V_2'$ of size $\eps^{1/3}n$ such that no vertex in $U$ is adjacent to $v$. 
	
	Now we will remove $\eps^{1/3}n$ edges from the graph $H$ and insert the edges $\{v,u\}$ for every $u\in U$. In particular, we will first remove every edge in $B\setminus B'$ which is adjacent to $v$. By Lemma~\ref{lem:bounddegree} the number of such edges is $d_B(v)\leq 5\eps^{2/3}n<\eps^{1/3}n$. In addition $\eps^{1/3}n\le \eps^{1/4}n/3\le |M|/3\le |B|/3$, which is at most $|B'|/2$ by~\eqref{eq:bprime}.
	Therefore $0< \eps^{1/3}n-d_B(v)\leq |B'|/2$ and the remaining $\eps^{1/3}n-d_B(v)$ edges can be removed from $B'$ in accordance to Lemma~\ref{lem:lowerbowtie}. Denote the set of $\eps^{1/3}n$ removed edges by $D$.
	Removing the edges in $D$ destroys at least
	$$(\eps^{1/3}n-d_B(v))\frac{\badedge n}{8}-\frac{\eps^{2/3}n^3}{2}\stackrel{\mathrm{Lem.}~\ref{lem:bounddegree}}{\geq}\frac{\eps^{1/3}\badedge n^2}{30}$$
	bowties, where we use $\badedge \geq d_M(v)\geq \eps^{1/4}n$.
	
	Let $G$ be the graph obtained after removing the edges in $D$. Note that $|N_G(v)\cap V_1|=0$ and for $u\in V_2'$ we have by Lemma~\ref{lem:newupperbad} that $|N_G(u)\cap V_2'|\leq 80\badedge/n+1$. In addition, inserting edges into $M$ keeps these values unchanged. Thus, Lemma~\ref{lem:missing} is applicable for each of the $\eps^{1/3}n$ edges inserted, and thus at most
	$$\eps^{1/3}n(4\eps^{1/4}\badedge n+6n)$$
	bowties are created.
	Since $b\geq \eps^{1/4}n$
	$$\frac{\eps^{1/3}\badedge n^2}{30}>\eps^{1/3}n(4\eps^{1/4}\badedge n+6n),$$
	contradicting the minimality of $H$.
\end{proof}

Since $H=H'$ we have that $d_B(v)=d_{B'}(v)$. Thus, Lemma~\ref{lem:newupperbad} implies that the conditions of Lemma~\ref{lem:missing} are satisfied for any pair $\{u,v\}\in M$, enabling us to prove $M=\emptyset$.

\begin{lemma}\label{lem:completebip}
	When $\badedge \geq \eps^{1/4}n$ we have $M=\emptyset$.
\end{lemma}

\begin{proof}
	Assume for contradiction $M\neq\emptyset$ and let $\{u,v\}\in M$. 
	Since $|B'|=|B|\ge 2$, by Lemma~\ref{lem:lowerbowtie} there exists an edge $\{w_1,w_2\}$ such that removing this edge destroys at least
	$\badedge n/8-n/2$ bowties. In addition, Lemma~\ref{lem:missing} implies that we create at most $4\eps^{1/4}\badedge n+6n$ bowties when inserting $\{u,v\}$.
	Since $\badedge \geq \eps^{1/4}n$ we have
	$$\frac{\badedge n}{8}-\frac{n}{2}>4\eps^{1/4}\badedge n+6n,$$
	which is a contradiction.
\end{proof}

	We conclude this section with the proof of Proposition~\ref{lem:cbg}.

\begin{proof}[Proof of Proposition~\ref{lem:cbg}]
	The first statement follows trivially from Lemmas~\ref{lem:bipartite}, \ref{lem:cbgsparse} and \ref{lem:completebip}. 	
	
	Next we show that $\badedgep{1}\leq 4q+4$. 
	Assume for contradiction that $\badedgep{1}>4q+4$. Taking any two disjoint edges $e_1,e_2$ in $B_1$ and a vertex in $V_2$ creates a bowtie. In addition if $e_1,e_2\in B_1$ are adjacent then any pair of vertices in $V_2$ creates a bowtie. 
	Thus the number of bowties in the graph is at least
	$$\binom{4(q+1)}{2}\partition{2}\geq \binom{4(q+1)}{2}(1-2\eps)\frac{n}{2}>\frac{27}{8}(q+1)^2n>(q+1)^2\left(13n/4+13\right)$$
	for large enough $n$. This together with Lemma~\ref{lem:numbowties} contradicts the minimality of $H$.
	An analogous argument shows $\badedgep{2}\leq 4q+4$.
	
	Let $\partition{1}=\lceil n/2 \rceil +a$ and $\partition{2}=\lfloor n/2 \rfloor -a$. Without loss of generality assume $a\geq 0$ i.e.\ $\partition{1}\geq \partition{2}$.
	Then the total number of pairs between $V_1$ and $V_2$ is 
	$$|V_1||V_2|=\lceil n/2 \rceil \lfloor n/2 \rfloor-a^2+(\lfloor n/2 \rfloor-\lceil n/2 \rceil) a\le |E(T_2(n))|-a^2$$
	implying that $\badedgep{1}+\badedgep{2}=\badedge \geq q+1+a^2$. Together with $\badedgep{1}+\badedgep{2}\leq 8q+8=o(n^2)$ this implies $a=o(n)$ completing the proof.
\end{proof}

\section{Proof of Theorems~\ref{thm:subgraph} and \ref{thm:assymptotics}}\label{sec:opt}

\subsection{Proof of Theorem~\ref{thm:subgraph}}
Throughout this section let $H$ be a graph on $n$ vertices and $\ex(n,F)+q$ edges containing the minimal number of bowties.
In the previous section we proved Proposition~\ref{lem:cbg} stating that there exists a partition of the vertex set of $H$ into two sets $V_1,V_2$, such that every edge between the two sets is present. In addition, both sets contain approximately $n/2$ vertices and the number of edges spanned by each of these sets is small.

Once the partition of Proposition~\ref{lem:cbg} has been established, we only need to examine the structure of the graphs spanned by $V_1$ and $V_2$. Denote the edges spanned by $V_1$ and $V_2$ by $B_1$ and $B_2$, respectively. Set $B=B_1\cup B_2$ and $\badedge=|B|$. In addition, for $i=1,2$ let 
$$\partition{i}=|V_i|\quad \mbox{and} \quad \badedgep{i}=|B_i|.$$

We start by investigating the number of bowties in such graphs.
Let $\widetilde{H}$ be a graph on $V$ containing $\ex(n,F)+q$ edges such that the vertex set of $\widetilde{H}$ can be partitioned into two parts $\widetilde{V}_1,\widetilde{V}_2$ with $E(K(\widetilde{V}_1,\widetilde{V}_2))\subseteq E(\widetilde{H})$. Denote the set of edges spanned by $\widetilde{V}_1$ and $\widetilde{V}_2$ by $\widetilde{B}_1$ and $\widetilde{B}_2$, respectively, also let $\widetilde{B}=\widetilde{B}_1\cup \widetilde{B}_2$.

Recall that for any $v\in V$ and $E'\subseteq \binom{V}{2}$ $d_{E'}(v)$ denotes the degree of vertex $v$ in the graph $(V,E')$.
Since $\widetilde{H}$ contains a complete bipartite subgraph we can express a lower bound on the number of bowties found in $\widetilde{H}$ via an explicit formula. 
Recall that bowties are formed of two triangles, and note that any triangle in $\widetilde{H}$ must contain at least one edge in $\widetilde{B}$. In particular, we restrict ourselves to bowties formed from two triangles which have exactly one edge in $\widetilde{B}$ and thus two edges between $\widetilde{V}_1$ and $\widetilde{V}_2$. Any such bowtie belongs to one of the 3 types of bowties seen in Figure~\ref{fig:types}. Any pair of disjoint edges in $\widetilde{V}_1$ is contained in $|\widetilde{V}_{2}|$ bowties. 
In addition a pair of adjacent edges in $\widetilde{V}_1$ can be found in $|\widetilde{V}_{2}|(|\widetilde{V}_{2}|-1)$ bowties. 
Similarly, two disjoint edges in $\widetilde{V}_2$ create $|\widetilde{V}_{1}|$ bowties, while two adjacent edges create $|\widetilde{V}_{1}|(|\widetilde{V}_{1}|-1)$ bowties.
Finally any two edges, where one is spanned by $\widetilde{V}_1$ and the other is spanned by $\widetilde{V}_2$, are contained in $2(n-4)$ bowties. This implies that the total number of bowties in $\widetilde{H}$ is at least

\begin{align*}
	2&(n-4) |\widetilde{B}_{1}||\widetilde{B}_{2}| + \sum_{i=1}^2 \left(\sum_{v\in \widetilde{V}_i} \binom{d_{\widetilde{B}}(v)}{2}|\widetilde{V}_{3-i}|\,(|\widetilde{V}_{{3-i}}|-1) + \left(\binom{|\widetilde{B}_{i}|}{2}-\sum_{v\in \widetilde{V}_i} \binom{d_{\widetilde{B}(v)}}{2}\right)|\widetilde{V}_{3-i}|\right).
\end{align*}
After trivial simplifications, this lower bound can be rewritten as
\begin{equation}\label{eq:lownumbowties}
	\#F(\widetilde{H})\ge f((d_{\widetilde{B}}(v))_{v\in \widetilde{V}_1},(d_{\widetilde{B}}(v))_{v\in \widetilde{V}_2}),
\end{equation}
where $f$ is the following function.  It takes as input two sequences of non-negative integers, $(d_{1,1},\dots,d_{1,v_1^*})\in\mathbb{N}^{v_1^*}$ and $(d_{2,1},\dots,d_{2,v_2^*})\in\mathbb{N}^{v_2^*}$ such that $v_1^*+v_2^*=n$ and $b_i^*=\frac12 \sum_{j=1}^{v_i^*} d_{i,j}$ is an integer for $i=1,2$. Then the value of the function $f$ is defined as
\begin{equation}\label{eq:f}
	f((d_{1,j})_{j=1}^{v_1^*},(d_{2,j})_{j=1}^{v_2^*})=2(n-4)  b_1^*b_2^*+ \sum_{i=1}^2 \left(\sum_{j=1}^{v_i^*} \binom{d_{i,j}}{2} v_{3-i}^*(v_{3-i}^*-2) + \binom{b_{i}^*}{2}{v}_{3-i}^*\right).
\end{equation}

Note that the bound in~\eqref{eq:lownumbowties} is sharp if neither $\widetilde{V}_1$ nor $\widetilde{V}_2$ contains a triangle. 
On the other hand, the following lemma gives us some converse to the above inequality, by replacing $\widetilde{B}_i$ by a graph $B_i^*$  that has the same number of edges  and, additionally, is triangle-free and almost regular. 

\begin{lemma}\label{lem:trifree}
	The following holds for all sufficiently large $n$. Let $V_1^*\cup V_2^*=V$ be a partition of $V$ and for $i=1,2$ set $v_i^*=|V_i^*|$.	Let $\phi:V\rightarrow\mathbb{N}$ be a function such that, for $i=1,2$, ${b}_i^*=\sum_{v\in {V}_i^*}\phi(v)/2$ is an integer and ${b}_i^*\le ({v}_i^*)^2/16$. Then there is a graph $H^*$ with $n$ vertices and ${v}_1^*{v}_2^*+{b}_1^*+{b}_2^*$ edges such that
	\begin{equation}\label{eq:H*}
		\#F(H^*)\le
		f\big(\,(\phi(v))_{v\in V_1^*},\, (\phi(v))_{v\in V_2^*}\,\big).
	\end{equation}
	Furthermore, if for some $i=1,2$, two values of $\phi$ on $V_i^*$ differ by more than 1, then the inequality in~\eqref{eq:H*} is strict.
\end{lemma}
\begin{proof}
	For $i=1,2$, pick integers $d_{i,1},\dots,d_{i,{v}_i^*}\in \{\,\lfloor 2{b}_i^*/{v}_i^*\rfloor,\,\lceil 2{b}_i^*/{v}_i^*\rceil\,\}$ with sum $2{b}_i^*$. Lemma~\ref{lem:K3freeDeg} shows that there is a triangle-free graph $B_i^*$ on $V_i^*$ whose degree sequence is $(d_{i,1},\dots,d_{i,{v}_i}^*)$. 
	Let $H^*$ be obtained by adding every edge 
	between the graphs $B_1^*$ and $B_2^*$. Clearly, $H^*$ has the stated order and size while
	$\#F(H^*)= f((d_{1,j})_{j=1}^{{v}_1^*},(d_{2,j})_{j=1}^{{v}_2^*})$. Using the  convexity of the function $x\mapsto \binom{x}{2}$ on $\mathbbm N$, we see that this is at most the left-hand side of~\eqref{eq:H*}, as required. 
	
	The second part of the lemma follows since the function $x\rightarrow \binom{x}{2}$ is \emph{strictly} convex on $\mathbbm N$.
\end{proof}

If we let $\widetilde{H}$ be our extremal graph $H$ and let $\phi(v)$ be $d_B(v)$ for $v\in V$, then, in view of $v_i=(1/2+o(1))n=\Omega(n)$ and $b_i=o(n^2)$, Lemma~\ref{lem:trifree} applies, providing another 
extremal graph~$H^*$. Thus,
both~\eqref{eq:lownumbowties} and~\eqref{eq:H*} are equalities. This means that, for $i=1,2$, the part $V_i$ does not induce a triangle and
\begin{equation}\label{eq:degdiff}
	|d_B(u)-d_B(v)|\leq 1,\quad \mbox{for every $u,v\in V_i$}.
\end{equation}

For future reference, let us repeat the formula for the number of bowties in $H$:
\begin{equation}\label{eq:numbowties}
	\#F(H)=2(n-4) \badedgep{1}\badedgep{2} + \sum_{i=1}^2 \left(\sum_{v\in V_i} \binom{d_B(v)}{2}\partition{3-i}(\partition{3-i}-2) + \binom{\badedgep{i}}{2}\partition{3-i}\right).
\end{equation}

Note that the number of vertices which have degree $\lfloor 2\badedgep{i}/\partition{i}\rfloor$ or $\lceil 2\badedgep{i}/\partition{i}\rceil$ in the graph $(V,B)$ is uniquely determined by $\badedgep{i}$ and $\partition{i}$. Therefore \eqref{eq:numbowties} depends only on $\badedgep{1},\badedgep{2},\partition{1},\partition{2}$.
In fact, we only need to establish the values of these parameters which lead to the minimal number of bowties. However, there is some dependence between the parameters. We only require one parameter to track both part sizes. Let
\begin{equation}\label{eq:partsizes}
	\partition{1}=\lceil n/2 \rceil+a \quad \mbox{and} \quad \partition{2}=\lfloor n/2 \rfloor -a
\end{equation}
for some $a\in \mathbb{Z}$. Proposition~\ref{lem:cbg} implies that $a=o(n)$. Note that $\badedge=q+1+(\lceil n/2 \rceil-\lfloor n/2 \rfloor)a+a^2$, thus it suffices to determine one of $\badedgep{1},\badedgep{2}$, once the value of $a$ has been established.

Theorem~\ref{thm:subgraph} follows if we can show that $a=0$ when $n$ is even and that $a=0$ or $-1$ when $n$ is odd. We will show that if neither of these holds, then moving a vertex from the larger part to the smaller decreases the number of bowties. Ideally, we would move the vertex in such a way that the number of neighbours of the vertex within its part remains unchanged and every edge between the new parts is still present. 
This leaves the total number of edges spanned by the parts unchanged, but increases the number of edges between $V_1$ and $V_2$ as $|a|$ is reduced.
Thus, in order to leave the total number of edges unchanged, we need to remove some additional edges. 
Although the previous argument is only applicable if the degree of the vertex moved between the parts is even, it can be adapted to work for odd degrees as well. This is achieved by removing an edge before the vertex is moved, in particular an edge adjacent to the vertex about to be moved, resulting in an even degree for the vertex. Further superfluous edges are removed after the vertex has been moved.

We need to estimate the change in the number of bowties after moving a vertex between the parts and removing superfluous edges. Later we will see that the number of bowties destroyed by removing edges outnumbers the number of bowties created when moving the vertex.

In order to prove Theorem~\ref{thm:subgraph} we need three auxiliary lemmas, the proofs of which can be found in Section~\ref{sec:proofs}.
We first estimate the number of bowties destroyed when a well-chosen set of $k$ edges is removed from the graph $H$, after a vertex has been moved between its parts.

\begin{lemma}\label{lem:lowerbowtie1}
	Let $U\subseteq V$ such that $|U|\geq n-2$ and $\widehat{H}$ a graph created from $H$ by adding and removing edges, such that no edge is removed from $E(H[U])$. If $\badedge\geq 10$ for every $k\leq \badedge/3$ there exists $D\subseteq B\cap E(H[U])$ with $|D|=k$ such that the removal of all edges of $D$ from the graph $\widehat{H}$ destroys at least
	$$k\frac{\badedge n}{8}-\frac{n k^2}{2}$$
	bowties. 
\end{lemma}

The number of bowties created when moving a vertex depends on the difference of $\badedgep{1}$ and $\badedgep{2}$. In particular, the closer the two values are, the smaller the change in the number of bowties is. 
In the following lemma we provide an upper bound on the difference of $\badedgep{1}$ and $\badedgep{2}$.

\begin{lemma}\label{lem:smalldif}
	We have that $|\badedgep{1}-\badedgep{2}|\leq n/4+33(|a|+1)q/n$.
\end{lemma}

As a consequence, we can derive that the degrees in the graph $(V,B)$ are also close. By~\eqref{eq:degdiff}, this holds in a very strong form for vertices that are in the same part $V_i$. So the following lemma says something new, only if $w$ is not in the same part as $u$ or $v$.

\begin{lemma}\label{lem:nosmalldeg}
	If there exists a pair of vertices $u,v\in V$ such that $d_B(u)=d_B(v)$, then every vertex $w\in V$, but at most one, satisfies $d_B(w)\geq d_B(u)-1-900(|a|+1)(d_B(u)-1)/n$. 
\end{lemma}

Using Proposition~\ref{lem:cbg} and Lemmas~\ref{lem:lowerbowtie1}--\ref{lem:nosmalldeg} we prove Theorem~\ref{thm:subgraph}.

\begin{proof}[Proof of Theorem~\ref{thm:subgraph}] 
	Assume for contradiction that the theorem is false. 
	Then for each $\delta>0$ there exists a graph $H$ that violates Theorem~\ref{thm:subgraph}. Let $\delta$ be sufficiently small so that the order $n\ge 1/\delta$ of $H$ and the surplus $q=|E(H)|-\ex(n,F)\le \delta n^2$ satisfy all forthcoming inequalities. Since the constant $\delta$ given by our proof is very small, we do not write it explicitly nor try to optimise the dependencies between other constants. For notational convenience, we use asymptotic notation, e.g.\ writing $q=o(n^2)$. 
	
	Proposition \ref{lem:cbg} applies and gives that $H$ contains a spanning complete bipartite graph $K(V_1,V_2)$. Define $v_i=|V_i|$ for $i=1,2$. 
	Without loss of generality we may assume that $\partition{1}\ge \partition{2}$. 
	Let $a=v_1-\lceil n/2\rceil$. By Proposition \ref{lem:cbg}, we have that $a=o(n)$. Since $H$ is a counterexample, we have $a\ge1$.
	The larger part $V_1$ must contain at least one edge, otherwise selecting $U\subseteq V_1$ with $|U|=\partition{2}$ and moving the graph spanned by $V_2$ to $U$ strictly decreases the number of bowties by~\eqref{eq:numbowties}. 
	
	Recall that, by~\eqref{eq:degdiff}, all bad degrees inside a part differ by at most~$1$. Let $d$ be the maximal integer such that there exist two vertices $v,w\in V_1$ such that $d_B(v)=d_B(w)=d$. 
	Since $V_1$ contains at least one edge, we have $d>0$. 
	
	Note that each bad degree is at most $d+1+o(d)$. Indeed, this is true by~\eqref{eq:degdiff} if $V_2$ spans no edges. Otherwise, $V_2$ has two vertices $u',v'$ of the same positive degree $d'$ and the claim follows  in view of $a=o(n)$ from Lemma~\ref{lem:nosmalldeg} applied to these two vertices. Thus, for every $z\in V$, we have, for example, $d_B(z)\le 2d+1$. Since $\partition{1},\partition{2}\leq n$, we have
	\begin{equation}\label{eq:orderb}
		b_1,b_2=O(dn).
	\end{equation}
	
	We consider two cases depending on the parity of $d$.
	
	\medskip\noindent	{\bf Case 1:} $d$ is even.\medskip 
	
	\noindent	Roughly speaking, we move $v$ from $V_1$ to $V_2$ in such a way that for every vertex in $V\setminus \{v\}$ the number of neighbours within its part remains unchanged, but for $v$ the number of neighbours within its part changes from $d$ to $d-2$, i.e.\ an edge is removed from $B$. In addition, every edge between the new parts is present.
	More formally, set $V_1^*=V_1\setminus \{v\}$, $V_2^*=V_2\cup \{v\}$ and define $\phi:V\rightarrow{\mathbb{N}}$ such that
	for $u\in V\setminus \{v\}$ we have $\phi(u)=d_{B}(u)$ and $\phi(v)=d_B(v)-2$. Let $H^*$ be the graph provided by Lemma~\ref{lem:trifree} for this function~$\phi$. We have that
	\begin{align*}
		\#F(H^*)&\le 2(n-4)\left(\badedgep{1}-\frac{d_B(v)}{2}\right)\left(\badedgep{2}+\frac{d_B(v)-2}{2}\right)+\binom{d_B(v)-2}{2}(\partition{1}-1)(\partition{1}-3)\\
		&+ \binom{\badedgep{1}-d_B(v)/2}{2}(\partition{2}+1)+ \binom{\badedgep{2}+(d_B(v)-2)/2}{2}(\partition{1}-1)\\
		&+ \sum_{u\in V_1\setminus\{v\}} \binom{d_B(u)}{2}(\partition{2}+1)(\partition{2}-1) +\sum_{u\in V_2} \binom{d_B(u)}{2}(\partition{1}-1)(\partition{1}-3).
	\end{align*}
	
	By \eqref{eq:numbowties} the number of bowties in $H$ is
	\begin{align*}
		\#F(H)&=2(n-4) \badedgep{1}\badedgep{2} +\binom{d_B(v)}{2}\partition{2}(\partition{2}-2)\\
		&+\binom{\badedgep{1}}{2}\partition{2}+ \binom{\badedgep{2}}{2}\partition{1}\\
		&+ \sum_{u\in V_1\setminus\{v\}} \binom{d_B(u)}{2}\partition{2}(\partition{2}-2) +\sum_{u\in V_2} \binom{d_B(u)}{2}\partition{1}(\partition{1}-2).
	\end{align*}
	
	We are interested in a lower bound on $\#F(H)-\#F(H^*)$, more precisely the difference of the above two bounds. 
	We will examine the difference of each of the six terms in order to determine the overall change. Recall that $d_B(v)=d$.
	
	We start with
	\begin{align}
		2(n-4) \badedgep{1}\badedgep{2}-2(n-4)\left(\badedgep{1}-\frac{d}{2}\right)\left(\badedgep{2}+\frac{d-2}{2}\right)\nonumber
		&=2(n-4)\left(\badedgep{2}\frac{d}{2}-\badedgep{1}\frac{d-2}{2}+\frac{d(d-2)}{4}\right)\\
		&\stackrel{\eqref{eq:orderb}}{\geq} nd\badedgep{2}-n(d-2)\badedgep{1}+O(d^2n).\label{eq:changefirst}
	\end{align}
	Recall that $\partition{1},\partition{2}=n/2+O(a)$. Thus for the following term we have
	\begin{align}
		\binom{d}{2}\partition{2}(\partition{2}-2)-\binom{d-2}{2}(\partition{1}-1)(\partition{1}-3)&=\left(\binom{d}{2}-\binom{d-2}{2}\right)\left(\frac{n}{2}\right)^2+O(a d^2 n)\nonumber\\
		&=(2d-3)\left(\frac{n}{2}\right)^2+O(a d^2 n).\label{eq:changemovedvertex}
	\end{align}
	
	By the definition of $d$ and~\eqref{eq:degdiff}, every vertex in $V_1$, except at most one, has degree at most $d$. Therefore $\badedgep{1}-1\leq dn/4+O(ad)$ and thus we have
	\begin{align}
		\binom{\badedgep{1}}{2}\partition{2}-\binom{\badedgep{1}-d/2}{2}(\partition{2}+1)
		&\geq \badedgep{1}\frac{d}{2}\partition{2}-\binom{\badedgep{1}}{2}+O(d^2n)\nonumber\\
		&= \badedgep{1}\frac{d}{2}\partition{2}-\badedgep{1}\frac{\badedgep{1}-1}{2}+O(d^2n)\nonumber\\
		&\stackrel{\eqref{eq:orderb}}{\geq}\badedgep{1}\left(\frac{d}{2}\frac{n}{2}-\frac{dn}{8}\right)+O(ad^2n)\nonumber\\
		&= \badedgep{1}\frac{dn}{8}+O(ad^2n).\label{eq:changeedgesp1}
	\end{align}
	Lemma~\ref{lem:nosmalldeg} implies that for every $u\in V_2$, except at most one, we have $d_B(u)\geq d-1+O(ad/n)$. Thus $\badedgep{2}-1\geq (d-1)n/4+O(ad)$ and similarly as before
	\begin{align}
		\binom{\badedgep{2}}{2}\partition{1}-\binom{\badedgep{2}+(d-2)/2}{2}(\partition{1}-1)
		&=-\badedgep{2}\frac{d-2}{2}\frac{n}{2}+\binom{\badedgep{2}}{2}+O(ad^2n)\nonumber\\
		&=-\badedgep{2}\frac{d-2}{2}\frac{n}{2}+\badedgep{2}\frac{\badedgep{2}-1}{2}+O(ad^2n)\nonumber\\
		&\geq -\badedgep{2}\frac{(d-3)n}{8}+O(ad^2n)\label{eq:changeedgesp2}.
	\end{align}

	All that is left to estimate are the two sums. The first of these is
	\begin{equation*}
		\sum_{u\in V_1\setminus\{v\}} \binom{d_B(u)}{2}\partition{2}(\partition{2}-2)-\sum_{u\in V_1\setminus\{v\}} \binom{d_B(u)}{2}(\partition{2}+1)(\partition{2}-1)
		=(-2\partition{2}+1)\sum_{u\in V_1\setminus\{v\}} \binom{d_B(u)}{2}.
	\end{equation*}

	Recall that for every $u\in V_1 \setminus\{v\}$, except at most one, we have that $d_B(u)\leq d$ and by \eqref{eq:degdiff} for the one exception we have $d_B(u) \leq d+1$. When $d_B(u)\leq d$, we have
	$$\binom{d_B(u)}{2}\leq \frac{d_B(u)(d-1)}{2},$$
	on the other hand, if $d_B(u)= d+1$, then
	$$\binom{d_B(u)}{2}=\frac{d_B(u)(d-1)}{2}+\frac{d+1}{2}.$$	
	Thus 
	\begin{align}
		(-2\partition{2}+1)\sum_{u\in V_1\setminus\{v\}} \binom{d_B(u)}{2}
		&\geq (-2\partition{2}+1)\left(\frac{d+1}{2}+\sum_{u\in V_1\setminus\{v\}} \frac{d_B(u)(d-1)}{2}\right)\nonumber\\
		&=(-2\partition{2}+1)(d-1)\sum_{u\in V_1\setminus\{v\}} \frac{d_B(u)}{2}+O(ad^2n)\nonumber\\
		&\geq (-2\partition{2}+1)(d-1)\badedgep{1}+O(ad^2n)\nonumber\\
		&\stackrel{\eqref{eq:orderb}}{=}-\badedgep{1}(d-1)n+O(ad^2n).\label{eq:changedeg}
	\end{align}
	
	Finally, we have
	\begin{equation*}
		\sum_{u\in V_2} \binom{d_B(u)}{2}\partition{1}(\partition{1}-2)-\sum_{u\in V_2} \binom{d_B(u)}{2}(\partition{1}-1)(\partition{1}-3)=(2\partition{1}-3)\sum_{u\in V_2} \binom{d_B(u)}{2}.
	\end{equation*}
	
	Recall that for every $u\in V_2$, except at most one, we have $d_B(u)\geq d-1-O(ad/n)$ and by \eqref{eq:degdiff} for the one exception we have $d_B(u)\geq d-2-O(ad/n)$, thus
	\begin{align}
		(2\partition{1}-3)\sum_{u\in V_2} \binom{d_B(u)}{2}
		&\geq (2\partition{1}-3)(d-2)\sum_{u\in V_2} \frac{d_B(u)}{2}+O(ad^2n)\nonumber\\
		&= \badedgep{2}(d-2)n+O(ad^2n).\label{eq:changelast}
	\end{align}
	
	Combining \eqref{eq:changefirst}--\eqref{eq:changelast} we have
	\begin{align}
		\#F(H)-\#F(H^*)&\geq(2d-3)\left(\frac{n}{2}\right)^2+\left(\frac{15d-13}{4}\badedgep{2}-\frac{15d-24}{4}\badedgep{1}\right)\frac{n}{2}+O(ad^2 n)\nonumber\\
		&\stackrel{\mathrm{Lem.}~\ref{lem:smalldif}}{\geq}(2d-3)\left(\frac{n}{2}\right)^2-\frac{15d-24}{8}\left(\frac{n}{2}\right)^2+O(ad^2 n)\nonumber\\
		&\stackrel{d=o(n)}{=} \frac{d}{8}\left(\frac{n}{2}\right)^2+o(a d n^2) \label{eq:resulte}
	\end{align}
	when $n$ is large enough. 
	
	Note that the number of edges in $H^*$ exceeds the number of edges in $H$ by
	\begin{align*}
	E(H^*)-E(H)&\geq \left(\lceil n/2\rceil+(a-1)\right)\left(\lfloor n/2\rfloor-(a-1)\right)-\left(\lceil n/2\rceil+a\right)\left(\lfloor n/2\rfloor-a\right)-1\\
	&\geq a^2-(a-1)^2-1= 2(a-1).
	\end{align*}
	When $a< 10$,
	$$\#F(H)-\#F(H^*)\stackrel{\eqref{eq:resulte}}{\geq}\frac{d}{8}\left(\frac{n}{2}\right)^2+o(adn^2)>0$$
	for large enough $n$ and thus already $H^*$ has fewer bowties than $H$ and removing the additional edges only decreases this number. On the other hand, if $a\geq 10$, then removing these $2(a-1)$ additional edges plays a significant role, because by Lemma~\ref{lem:lowerbowtie1} this destroys at least 
	$(a-1)\badedge n/4-2(a-1)^2n$ bowties. Recall that $\badedge\geq q+1+a^2$ and when $a\geq 10$ then $b\geq q+1+a^2\geq 10(a-1)$ leading to the destruction of at least
	\begin{equation}\label{eq:edgeadj}
		(a-1)n\left(\frac{b}{4}-2(a-1)\right)\geq (a-1)n\left(\frac{b}{4}-\frac{2b}{10}\right)= (a-1)\frac{\badedge}{20}n\stackrel{b\geq (d-1)n/2,d\geq 2}{\geq}\frac{adn^2}{100}
	\end{equation}
	bowties.
	This together with \eqref{eq:resulte} leads to a contradiction.\medskip
	
	\noindent {\bf Case 2:} $d$ is odd.\medskip
	
	\noindent Recall that $v,w\in V_1$ are such that $d_B(v)=d_B(w)=d$. Roughly speaking, we will move $v$ from $V_1$ to $V_2$ in such a way that for every vertex in $V\setminus \{v,w\}$ the number of neighbours within its part remains unchanged, but for $v$ and $w$ the number of neighbours within its part changes from $d$ to $d-1$. 
	Formally, set $V_1^*=V_1\setminus \{v\}$, $V_2^*=V_2\cup \{v\}$ and define $\phi:V\rightarrow{\mathbb{N}}$ such that $\phi(v)=d_B(v)-1$, $\phi(w)=d_B(w)-1$ and
	for $u\in V\setminus \{v,w\}$ we have $\phi(u)=d_{B}(u)$.
	Let $H^*$ be the graph returned by  Lemma~\ref{lem:trifree}.
	The number of bowties in $H^*$ satisfies
	\begin{align*}
		\#F(H^*)&\le 2(n-4)\left(\badedgep{1}-\frac{d_B(v)+1}{2}\right)\left(\badedgep{2}+\frac{d_B(v)-1}{2}\right)\\
		&+\binom{d_B(v)-1}{2}(\partition{1}-1)(\partition{1}-3)+\binom{d_B(w)-1}{2}(\partition{2}+1)(\partition{2}-1)\\
		&+ \binom{\badedgep{1}-(d_B(v)+1)/2}{2}(\partition{2}+1)+ \binom{\badedgep{2}+(d_B(v)-1)/2}{2}(\partition{1}-1)\\
		&+ \sum_{u\in V_1\setminus \{v,w\}} \binom{d_B(u)}{2}(\partition{2}+1)(\partition{2}-1) +\sum_{u\in V_2} \binom{d_B(u)}{2}(\partition{1}-1)(\partition{1}-3).
	\end{align*}		
	
	Clearly, the number of bowties in $H$ is the same as earlier
	\begin{align*}
		\#F(H)&=2(n-4) \badedgep{1}\badedgep{2}\\ &+\binom{d_B(v)}{2}\partition{2}(\partition{2}-2)+\binom{d_B(w)}{2}\partition{2}(\partition{2}-2)\\
		&+\binom{\badedgep{1}}{2}\partition{2}+ \binom{\badedgep{2}}{2}\partition{1}\\
		&+ \sum_{u\in V_1\setminus \{v,w\}} \binom{d_B(u)}{2}\partition{2}(\partition{2}-2) +\sum_{u\in V_2} \binom{d_B(u)}{2}\partition{1}(\partition{1}-2).
	\end{align*}

	The calculations, for most part, are analogous to Case 1. In \eqref{eq:changefirst} we just need to replace $d$ with $d+1$
	\begin{align}
		2(n-4) \badedgep{1}\badedgep{2}-2(n-4)\left(\badedgep{1}-\frac{(d+1)}{2}\right)\left(\badedgep{2}+\frac{d-1}{2}\right)\nonumber
		&\geq n(d+1)\badedgep{2}-n(d-1)\badedgep{1}+O(d^2n).\label{eq:schangefirst}
	\end{align}
	
	On the other hand, in \eqref{eq:changemovedvertex} the vertex degree decreases only by one, leading to		
	\begin{align*}
		\binom{d}{2}\partition{2}(\partition{2}-2)-\binom{d-1}{2}(\partition{1}-1)(\partition{1}-3)
		&=\binom{d}{2}\left(\frac{n}{2}\right)^2-\binom{d-1}{2}\left(\frac{n}{2}\right)^2+O(ad^2n)\\
		&=(d-1)\left(\frac{n}{2}\right)^2+O(a d^2 n).
	\end{align*}
	
	However, this time the degree of another vertex in $V_1$ also decreases
	\begin{align*}
		\binom{d_B(w)}{2}\partition{2}(\partition{2}-2)-\binom{d_B(w)-1}{2}(\partition{2}+1)(\partition{2}-1)
		&=(d-1)\left(\frac{n}{2}\right)^2+O(ad^2n).
	\end{align*}
	
	For \eqref{eq:changeedgesp1} and \eqref{eq:changeedgesp2} our bounds on $\badedgep{1}$ and $\badedgep{2}$ still hold, leading to
	\begin{align*}
		\binom{\badedgep{1}}{2}\partition{2}-\binom{\badedgep{1}-(d+1)/2}{2}(\partition{2}+1)
		&\geq\badedgep{1}\left(\frac{d+1}{2}\frac{n}{2}-\frac{dn}{8}\right)+O(ad^2n)\nonumber\\
		&= \badedgep{1}\frac{(d+2)n}{8}+O(ad^2n)
	\end{align*}
	and
	\begin{align*}
		\binom{\badedgep{2}}{2}\partition{1}-\binom{\badedgep{2}+(d-1)/2}{2}(\partition{1}-1)
		&\geq-\badedgep{2}\frac{d-1}{2}\frac{n}{2}+\badedgep{2}\frac{(d-1)n}{8}+O(ad^2n)\nonumber\\
		&\geq -\badedgep{2}\frac{(d-1)n}{8}+O(ad^2n).
	\end{align*}
	
	Finally, note that removing an additional vertex from the sum in \eqref{eq:changedeg} has no affect on the lower bound and \eqref{eq:changelast} remains unchanged. Therefore, we have
	\begin{align*}
		\#F(H)-\#F(H^*)
		&\geq 2(d-1)\left(\frac{n}{2}\right)^2+\left(\frac{15d-7}{4}\badedgep{2}-\frac{15d-18}{4}\badedgep{1}\right)\frac{n}{2}+O(ad^2n)\\
		&\stackrel{\mathrm{Lem.}~\ref{lem:smalldif}}{\geq} 2(d-1)\left(\frac{n}{2}\right)^2-\frac{15d-18}{8}\left(\frac{n}{2}\right)^2+O(ad^2n)\\
		&\geq \frac{d+2}{8}\left(\frac{n}{2}\right)^2+O(ad^2n)
	\end{align*}
	when $n$ is large enough. Similarly as before we have at least $2(a-1)$ additional edges in the graph. If $a<10$ or $d=1$, then the number of bowties has already decreased even before removing these edges. In the remaining cases, a calculation analogue to \eqref{eq:edgeadj} implies that the removal of the additional edges decreases the number of bowties.
\end{proof}

\subsection{Proof of Theorem~\ref{thm:assymptotics}}
In order to prove Theorem~\ref{thm:assymptotics} we need to show
$$h_F(n,q)=(1\pm c)\frac{n}{2}\left[\binom{e_1}{2}+\binom{e_2}{2}+m\binom{d+1}{2}\frac{n}{2}+(n-m)\binom{d}{2}\frac{n}{2}+4 e_1 e_2\right],$$	
where	
$$e_1=\left\lfloor \frac{dn}{4}+\frac{\min\{m, n/2\}}{2}\right\rfloor \quad \mbox{and} \quad e_2=q+1-e_1.$$

In the previous subsection we established the values of $\partition{1}$ and $\partition{2}$, thus we only need to determine the values of $\badedgep{1}$ and $\badedgep{2}$. More precisely, the asymptotics of these two values suffice, which we achieve by analysing the degree sequence in the graphs spanned by $V_1$ and $V_2$. We first show that the bad degree of almost every vertex must take one of two values (Lemma~\ref{lem:degreeposs}). In addition in one of the partitions almost every vertex must have the same bad degree (Lemma~\ref{lem:partitiondegree}). The proofs of these two lemmas can be found in Section~\ref{sec:proofs}. Lemmas~\ref{lem:degreeposs} and \ref{lem:partitiondegree} are sufficient to establish the asymptotics of $\badedgep{1}$ and $\badedgep{2}$ and complete the proof of Theorem~\ref{thm:assymptotics}.

\begin{lemma}\label{lem:degreeposs}
	Let $d=\lfloor 2(q+1)/n \rfloor$. For every vertex $v\in V$, we have $d-1\leq d_B(v)\leq d+2$. In addition, both the number of vertices with bad degree $d-1$ and $d+2$ is at most one.
\end{lemma}

This lemma shows that almost every vertex has degree $\lfloor 2(q+1)/n \rfloor$ or $\lfloor 2(q+1)/n \rfloor+1$ in the graph $(V,B)$. We have yet to establish how many of these vertices are contained in $V_1$ and $V_2$. In the following lemma we show that in one of the two parts almost every vertex has the same degree. In fact, we show a more general statement. By~\eqref{eq:degdiff}, pick integers $k$ and $\ell$ such that every vertex in $V_1$ has degree $k$ or $k+1$ and every vertex in $V_2$ has degree $\ell$ or $\ell+1$.
Let $C_{i}$ denote the set of vertices of degree $i$ in $V_1$ and $D_i$ denote the set of vertices of degree $i$ in $V_2$.

\begin{lemma}\label{lem:partitiondegree}
	In $H$, at least one of $|C_{k}|,|C_{k+1}|,|D_\ell|,|D_{\ell+1}|$ is at most 1.
\end{lemma}

The previous two lemmas give us a satisfactory estimate for the degree sequence of $(V,B)$, allowing us to determine the asymptotics of $h_F(n,q)$.

\begin{proof}[Proof of Theorem~\ref{thm:assymptotics}]
	By \eqref{eq:numbowties} and since $\partition{i}=(1+o(1))n/2$, the number of bowties is
	$$(1+o(1))\frac{n}{2}\left[ \binom{\badedgep{1}}{2}+ \binom{\badedgep{2}}{2}+ \frac{n}{2}\sum_{v\in V} \binom{d_B(v)}{2} + 4\badedgep{1}\badedgep{2}\right].$$
	When $q<\lfloor n/4 \rfloor-1$, 
	Lemma~\ref{lem:partitiondegree} implies that the optimal solution is when $V_1$ or $V_2$ contains a matching of size $q+1$ and the other part does not span any edge. Therefore the statement holds for this range of~$q$. 
	
	Now assume that $q\geq \lfloor n/4 \rfloor -1$. 
	Recall from the statement of the theorem that $2(q+1)=dn+m$ for $d,m\in\mathbb{N}$ with $m<n$ and $d=\lfloor 2(q+1)/n \rfloor$.
	By Lemma~\ref{lem:degreeposs} we have that almost every vertex has degree $d$ or $d+1$ in $(V,B)$ and any other vertex has degree $d-1$ or $d+2$. Thus we obtain
	$$\sum_{v\in V} \binom{d_B(v)}{2}=(n-m)\binom{d}{2}+m\binom{d+1}{2}+O(d+1).$$
	
	Since $d=o(n)$, we have $(d+1)n=o(n^2)$. In addition, $q=\Omega(n)$ and thus $\badedgep{1}^2+\badedgep{2}^2=\Omega(n^2)$. Therefore, the total number of bowties is
	$$(1+o(1))\frac{n}{2}\left[ \binom{\badedgep{1}}{2}+ \binom{\badedgep{2}}{2}+ \frac{n}{2}\left((n-m)\binom{d}{2}+m\binom{d+1}{2}\right) + 4\badedgep{1}\badedgep{2}\right].$$

	All that is left to show is that one of the parts contains $(1+o(1))e_1$ edges. 
	Recall that the sum of the bad degrees is $dn+m$. Let $m>n/2$. By Lemma~\ref{lem:degreeposs} there exist at least $n/2-3$ vertices of bad degree $d+1$. If each part has at least 2 vertices with bad degree $d+1$, then by Lemma~\ref{lem:partitiondegree} all but at most one vertex must have bad degree $d+1$ in one of the parts. Together with Theorem~\ref{thm:subgraph} this implies that there exists a part with at least $n/2-2$ vertices of bad degree $d+1$. Otherwise, one part has at most one vertex of bad degree $d+1$ and thus the other must have at least $n/2-4$ vertices of bad degree $d+1$.
	So in either case we have a part containing at least $n/2-4$ vertices of bad degree $d+1$ and without loss of generality we may assume that this is $V_1$. Then $b_1=(1+o(1))(d+1)n/2=(1+o(1))e_1$.
	
	Now consider the case when $m\leq n/2$. A similar argument as before, implies the existence of a part containing $n/2-4$ vertices with bad degree $d$ and without loss of generality assume that this is $V_2$. Therefore,
	$$b_1=\frac{dn+m}{2}-(1+o(1))\frac{dn}{4}-o\left(\frac{(d+1)n}{4}\right)=(1+o(1))e_1$$
	as required.
\end{proof}

\textbf{Remark}
Lemmas~\ref{lem:degreeposs} and \ref{lem:partitiondegree} leave only a limited number of options for the degree sequence of the extremal graph $H$. The exact value of $h_F(n,q)$ can be deduced after a precise analysis. In particular when $n$ is divisible by 4, since the number of vertices in both parts is even the aforementioned lemmas imply that every vertex has bad degree $d$ or $d+1$ and every vertex in one of the parts has the same bad degree. Therefore the formula in Theorem~\ref{thm:assymptotics} holds exactly, i.e.\ without the $1\pm c$ multiplier.

\section{Proof of technical lemmas}\label{sec:proofs}

	We start with results on the existence of triangle-free graphs with a given degree sequence, culminating in Lemma~\ref{lem:K3freeDeg} which was used in the proof of  Lemma~\ref{lem:trifree}. While some of the intermediate steps can be derived from the Gale-Ryser theorem \cite{Gale57,Ryser57} that characterises possible degree sequences of bipartite graphs, we present simple direct constructions instead. Let an \emph{$(\alpha,a,\beta,b)$-graph} mean a triangle-free graph with $\alpha+\beta$ vertices which has $\alpha$ vertices of degree $a$ and $\beta$ vertices of degree $b$ if $a\not=b$ and is $a$-regular if $a=b$. Trivially an $(\alpha,a,\beta,b)$ graph is also a $(\beta,b,\alpha,a)$ graph.

	\begin{lemma}\label{lem:trifreeeven}
	For any non-negative integers $d,i,m$ satisfying $d< i+m$ there exists a $(2i,d+1,2m,d)$-graph. In addition this is a balanced bipartite graph.
	\end{lemma}
	
	\begin{proof}
	Partition the vertices into two sets $\{v_{1,j}:j=1,\ldots, i+m\}$ and $\{v_{2,j}:j=1,\ldots, i+m\}$. Join $v_{1,j}$ with $v_{2,j}$ by an edge for every $j\leq i$. Let $M_\ell$ be a perfect matching where for every $j=1,\ldots,i+m$ the vertex $v_{1,j}$ is connected to $v_{2,k}$ where $k \equiv j+\ell \pmod{i+m}$. In order to complete the graph, insert $M_\ell$ for $\ell=1,\ldots, d$. The obtained graph has the desired degree sequence and is also bipartite (and thus  triangle-free).
	\end{proof}

	\begin{lemma}\label{lem:trifreeodd}
	For any integers $k,m$ satisfying $k\geq m\geq 0$ there exists a $(4k,k,1,2m)$-graph.
	\end{lemma}
	
	\begin{proof}
	Denote by $u$ the vertex of degree $2m$ and partition the remaining vertices into 4 sets of size $k$: $U_1,\ldots,U_4$. Denote the vertices in $U_i$ by $u_{i,j}$ for $j=1,\ldots,k$. Join $u$ to vertices $u_{i,j}$ by an edge, where $i=1,3$ and $j=1,\ldots,m$, i.e.\ we join $u$ to $m$ vertices in $U_1$ and to $m$ vertices in $U_3$. In addition, for $j=m+1,\ldots, k$ insert an edge between $u_{1,j}$ and $u_{3,j}$. Also for $j=1,\ldots,k$ insert an edge between $u_{2,j}$ and $u_{4,j}$. Finally insert every edge between $U_{1}$ and $U_2$ and every edge between $U_3$ and $U_4$ except the ones in the set $\{\{u_{i,j},u_{i+1,j}\}:i=1,2, j=1,\ldots,k\}$. 
	
	Note that the obtained graph $G$ has the required degree sequence. 
	In addition $u$ is not contained in a triangle, while, $G-u$ is bipartite with parts $U_1\cup U_4$ and $U_2\cup U_3$. Thus, $G$ is triangle-free, as required.
	\end{proof}

\begin{lemma}\label{lem:K3freeDeg}
 Let $a,b,\alpha,\beta$ be non-negative integers such that $|a-b|=1$, $\alpha a+\beta b$ is even and $3a+3b< \alpha+\beta-1$. Then there is an $(\alpha,a,\beta,b)$-graph.
 \end{lemma}

\begin{proof} If both $\alpha$ and $\beta$ are even, then  Lemma~\ref{lem:trifreeeven} directly gives the desired graph. Assume that at least one of $\alpha$ and $\beta$ is odd. In fact, exactly one of them is odd, because $a$ and $b$ have different parities while $\alpha a+\beta b$ is even by our assumption. By symmetry, assume that $\alpha$ is odd and $\beta$ is even. Then, necessarily, $a$ is even. 

By the remaining assumption of the lemma, we have that $6a< \alpha+\beta-1$ or $6b< \alpha+\beta-1$. In the former case we take the disjoint union of the $(4a,a,1,a)$-graph given by Lemma~\ref{lem:trifreeodd} (with $k=a$ and $m=a/2$) and the $(\alpha-4a-1,a,\beta,b)$-graph given by Lemma~\ref{lem:trifreeeven}. In the latter case we take the $(4b,b,1,a)$-graph of Lemma~\ref{lem:trifreeodd} (with $k=b$ and $m=a/2$) and the $(\alpha-1,a,\beta-4b,b)$-graph of  Lemma~\ref{lem:trifreeeven}.
\end{proof}

Now we provide the proofs of the auxiliary lemmas of Section~\ref{sec:opt}.

\begin{proof} [Proof of Lemma~\ref{lem:lowerbowtie1}]
	The number of bowties destroyed by removing the edges in $D$ from the graph $\widehat{H}$, is at least the number of bowties destroyed by removing $D$ from the graph $H[U]$. We will analyse the latter.	
	
	Note that
	$$|B\cap E(H[U])|\stackrel{\mathrm{Lem.~\ref{lem:newupperbad}}}{\geq}\badedge-2\left(\frac{80\badedge}{n}+1\right)\geq \frac{2\badedge}{3},$$
	where the last inequality follows from our assumption that $\badedge\geq 10$ and because $n$ is large enough.

	For simplicity of notation let $U_1=U\cap V_1$ and $U_2=U\cap V_2$. Without loss of generality we assume that $U_1$ spans at least as many edges in $B\cap E(H[U])$ as $U_2$. Select an arbitrary pair of edges $\{w_1,w_2\},\{w_3,w_4\}\in B\cap E(H[U_1])$. If $\{w_3,w_4\}$ is disjoint from $\{w_1,w_2\}$, then for every shared neighbour of $w_1,w_2,w_3,w_4$ we have a bowtie. On the other hand, if $\{w_3,w_4\}$ is adjacent to $\{w_1,w_2\}$, then every pair of shared neighbours results in a bowtie. By Proposition~\ref{lem:cbg} every edge between $U_1$ and $U_2$ is present. Since only two vertices were removed from the graph we have $|U_1|,|U_2|=(1+o(1))n/2$. Thus the number of bowties containing $\{w_1,w_2\}$ and $\{w_3,w_4\}$ is at least
	$(1+o(1))n/2.$
	
	Therefore, removing arbitrary $k$ edges from $B\cap E(H[U_1])$ destroys at least
	\begin{align*}
	(1+o(1))\frac{n}{2}\left(k\left(\frac{|B\cap E(H[U])|}{2}-k\right)\right)
	\geq (1+o(1))\frac{n}{2}\left(k\left(\frac{\badedge }{3}-k\right)\right)
	\geq k\frac{\badedge n}{8}-\frac{n k^2}{2}
	\end{align*}
	bowties.
\end{proof}

\begin{proof}[Proof of Lemma~\ref{lem:smalldif}]
	Assume for contradiction that $|\badedgep{1}-\badedgep{2}|> n/4+33(|a|+1)q/n$. Without loss of generality assume that 
	\begin{equation}\label{eq:manybadedges}
	\badedgep{1}>\badedgep{2}+n/4+33 (|a|+1)q/n.
	\end{equation}
	
	Roughly speaking, we move $e_M:=\lfloor n/4 \rfloor-|a|$ edges from $V_1$ to $V_2$ and show that the resulting graph has fewer copies of bowties. 
	Let $R_1$ be a set of vertices in $V_1$ of size $2\lfloor n/4 \rfloor-2|a|$ maximising the function $\sum_{v\in R_1}d_B(v)$, i.e.\ it contains the $2\lfloor n/4 \rfloor-2|a|$ vertices $v\in V_1$ with the largest value of $d_B(v)$. Note that
	$$|V_1|-|R_1|\geq \left\lfloor\frac{n}{2}\right\rfloor -|a|-2\left\lfloor\frac{n}{4}\right\rfloor +2|a|\geq |a|\geq 0.$$
	On the other hand, let $R_2$ be the set of vertices in $V_2$ of size $2\lfloor n/4 \rfloor-2|a|$ minimising the function $\sum_{v\in R_2}d_B(v)$. 
	In order to move the $e_M$ edges we reduce the value of $d_B(v)$ for $v\in R_1$ by one and increase the value of $d_B(v)$ for $v\in R_2$ by one. 
	Define $\phi:V\rightarrow \mathbb{N}$ by
$$	\phi(v) =
	\left\{
	\begin{array}{ll}
		d_B(v)-1,  & \mbox{if } v\in R_1, \\
		d_B(v), & \mbox{if } v\in V\setminus(R_1\cup R_2),\\
		d_B(v)+1, & \mbox{if } v\in R_2.
	\end{array}
	\right.
$$	
	
Since by \eqref{eq:degdiff} the value of $d_B(v)$ differs by at most one for vertices in $V_1$ and $R_1$ contains the vertices $v$ with the largest value of $d_B(v)$, we also have that the value of $\phi(v)$ differs by at most one for vertices in $V_1$. The same argument also gives us that the value of $\phi(v)$ differs by at most one for vertices in $V_2$.

Denote by $H^*$ the  graph returned by Lemma~\ref{lem:trifree} for $V_1^*=V_1$ and $V_2^*=V_2$.
		By \eqref{eq:numbowties} and Lemma~\ref{lem:trifree} we have
	\begin{align*}
	\#F(H)-\#F(H^*)
	&\ge \sum_{u\in R_1} \phi(u) \partition{2}(\partition{2}-2)+e_M \badedgep{2}(2n-8)+\partition{2}\left(e_M\badedgep{1}-\frac{e_M^2}{2}-\frac{e_M}{2}\right)\\
	&-\sum_{u\in R_2} d_{B}(u) \partition{1}(\partition{1}-2)-e_M (\badedgep{1}-e_M)(2n-8)-\partition{1}\left(\badedgep{2}e_M+\frac{e_M^2}{2}-\frac{e_M}{2}\right).
	\end{align*}
	By Proposition~\ref{lem:cbg} and \eqref{eq:manybadedges} we have $4q+4\geq \badedgep{1}\geq n/4$, implying $\badedgep{1}\leq (1+o(1))4q$ and $16q/n\geq 1+o(1)$. 
	Therefore $\sum_{u\in R_1}\phi(u)\leq \sum_{u\in V_1} \phi(u)\leq 2\badedgep{1}\leq (1+o(1))8q\leq 9q$. Thus
	$$\sum_{u\in R_1} \phi(u) \partition{2}(\partition{2}-2)\geq \sum_{u\in R_1} \phi(u) \left(\frac{n}{2}-|a|-\frac{1}{2}\right)\left(\frac{n}{2}-|a|-\frac{5}{2}\right) \geq \sum_{u\in R_1} \phi(u) \left(\frac{n}{2}\right)^2-14(|a|+1)qn.$$
	Proposition~\ref{lem:cbg} also implies that $\partition{1}=(1+o(1))n/2$, therefore the average vertex degree in $H[V_1]$ is $2\badedgep{1}/\partition{1}\leq (1+o(1))16q/n$. Together with \eqref{eq:degdiff}, for every vertex $v\in V_1$ we have $d_B(v)\leq(1+o(1))16q/n+1$ for large enough $n$. Recall that $16q/n \geq 1+o(1)$. Thus for $v\in V_1$ we have $d_B(v)\leq (1+o(1))32q/n\leq 33q/n$.
	Therefore, we have
	\begin{align*}
	\sum_{u\in R_1} \phi(u) \left(\frac{n}{2}\right)^2-14(|a|+1)qn
	&\geq \left[\sum_{u\in V_1} \phi(u)-|V_1\setminus R_1|\frac{33q}{n} \right]\left(\frac{n}{2}\right)^2-14(|a|+1)qn\\
	&\geq \left(\frac{n}{2}\right)^2\sum_{u\in V_1} \phi(u)-39(|a|+1)qn,
	\end{align*}
	as $|V_1\setminus R_1|\leq \lceil n/2 \rceil +|a| -2\lfloor n/4\rfloor +2|a|\leq 3(|a|+1)$.

	Since $q\geq (1+o(1)) n/16$ we have $\badedgep{2}\leq4q+4\leq 5q$ for large enough $n$. Together with $e_M\leq n/4$ this implies 
	\begin{equation*} 
	e_M \badedgep{2}(2n-8)\geq e_M \badedgep{2}2n-10(|a|+1)qn \stackrel{e_M\geq n/4-|a|-1}{\geq} \left(\frac{n}{2}\right)^2\sum_{u\in V_2} d_B(u)-20(|a|+1)qn.
	\end{equation*}
	
	Finally using $e_M= \lfloor n/4 \rfloor -|a|$ and $e_M+1\leq n/4+1\leq \badedgep{1}\leq (1+o(1))4q$ 
	we have
	\begin{align*}
	\partition{2}\left(e_M\badedgep{1}-\frac{e_M^2}{2}-\frac{e_M}{2}\right)
	&\geq\frac{n}{2}\left(e_M\badedgep{1}-\frac{e_M^2}{2}-\frac{e_M}{2}\right)-2(|a|+1)qn\\
	&\geq \frac{n}{2}b_1\left(\frac{n}{4}-|a|-1\right)-\frac{n}{4}\left(\left(\frac{n}{4}\right)^2+\frac{n}{4}\right)-2(|a|+1)qn\\
	&\geq\frac{n^2}{8}\badedgep{1} -\frac{n^3}{64} -\frac{n^2}{16} -5(|a|+1)qn.
	\end{align*}
	Therefore, we obtain
	\begin{align*}
	\sum_{u\in R_1} \phi(u) \partition{2}(\partition{2}-2)&+e_M \badedgep{2}(2n-8)+\partition{2}\left(e_M\badedgep{1}-\frac{e_M^2}{2}-\frac{e_M}{2}\right)\\
	&\geq\left(\frac{n}{2}\right)^2\sum_{u\in V_1} \phi(u)+\left(\frac{n}{2}\right)^2\sum_{u\in V_2} d_B(u)+\frac{n^2}{8}\badedgep{1}-\frac{n^3}{64}-\frac{n^2}{16}-64(|a|+1)qn.
	\end{align*}
	A similar argument shows that
	\begin{align*}
	\sum_{u\in R_2} d_{B}(u) \partition{1}(\partition{1}-2)&+e_M (\badedgep{1}-e_M)(2n-8)+\partition{1}\left(\badedgep{2}e_M+\frac{e_M^2}{2}-\frac{e_M}{2}\right)\\
	&\leq \left(\frac{n}{2}\right)^2\sum_{u\in V_2} d_B(u)+\left(\frac{n}{2}\right)^2\sum_{u\in V_1} \phi(u)+\frac{n^2}{8} \badedgep{2} +\frac{n^3}{64}+\frac{n^2}{16}+64(|a|+1)qn.
	\end{align*}
	Summing up, we have
	$$\#F(H)-\#F(H^*)\geq \frac{n^2}{8}\left[\badedgep{1}-\badedgep{2}-\frac{n}{4}-1-16\frac{(|a|+1)q}{n}\right]\stackrel{16 q/n\geq 1+o(1)}{\geq} \frac{n^2}{8}\left[\badedgep{1}-\badedgep{2}-\frac{n}{4}-33\frac{(|a|+1)q}{n}\right],$$
	which is positive due to \eqref{eq:manybadedges}, leading to a contradiction on minimality.
\end{proof}

\begin{proof}[Proof of Lemma~\ref{lem:nosmalldeg}]
	Let $d:=d_B(u)-1=d_B(v)-1$.
	We start with the case when $u,v\in V_1$, $d\geq 1$ and at most one vertex in $V_1$ has degree $d+2$ in $(V,B)$.
	
	Assume for contradiction that there exist 2 vertices $w_1,w_2\in V$ with $d_B(w_1),d_B(w_2)< d-900(|a|+1)d/n$. 
	By \eqref{eq:degdiff} for any vertex $w\in V_1$ we have that $d_B(w)\geq d$ so in fact $w_1,w_2\in V_2$.
	In addition, we may choose $w_1,w_2$ such that for every $w\in V_2\setminus \{w_1,w_2\}$ we have $d_B(w_1),d_B(w_2)\leq d_B(w)$.
	Let $\phi:V\rightarrow \mathbb{N}$ be the function defined by
	$$	\phi(z) =
	\left\{
	\begin{array}{ll}
	d_B(z)-1,  & \mbox{if }  z=u,v,\\
	d_B(z), & \mbox{if } v\in V\setminus\{u,v,w_1,w_2\},\\
	d_B(z)+1, & \mbox{if } z=w_1,w_2.
	\end{array}
	\right.
	$$	
	Let $H^*$ be the graph returned by Lemma~\ref{lem:trifree} for $V_1^*=V_1$ and $V_2^*=V_2$. Thus, by \eqref{eq:numbowties}, we have
	\begin{align*}
	\#F(H)-\#F(H^*)
	&\ge (d_B(u)+d_B(v)-2)\partition{2}(\partition{2}-2)+\badedgep{2} (2n-8)+(\badedgep{1}-1)\partition{2}\\
	&-(d_B(w_1)+d_B(w_2))\partition{1}(\partition{1}-2)-(\badedgep{1}-1)(2n-8)-\badedgep{2}\partition{1}.
	\end{align*}
	Note that $d_B(v),d_B(w_1)\leq d+1$ and together with \eqref{eq:degdiff} we have that $d_B(w)\leq d+2$ for any $w\in V$. Therefore, $\badedgep{1},\badedgep{2}\leq (1+o(1))(d+2)n/4\leq dn$ as $d\geq 1$. Thus we have
	\begin{align*}
	\#F(H)-\#F(H^*)
	&\geq(2d-d_B(w_1)-d_B(w_2))\frac{n^2}{4}+\frac{3}{2}n(\badedgep{2}-\badedgep{1})-26(|a|+1)dn.
	\end{align*}
	Recall that $d_B(w)\leq d+2$ for any $w\in V$. Therefore $q\leq (d+2)n/2\stackrel{d\geq 1}{\leq} 3dn/2. $
	This together with Lemma~\ref{lem:smalldif} implies $\badedgep{2}\leq \badedgep{1}+n/4+33(|a|+1)q/n\leq \badedgep{1}+n/4+50(|a|+1)d$ and thus
	\begin{equation}\label{eq:bowtiediff}
	\#F(H)-\#F(H^*)
	\geq(2d-d_B(w_1)-d_B(w_2))\frac{n^2}{4}-\frac{3}{2}n\,\frac{n}{4}-101(|a|+1)dn.
	\end{equation}

	We first consider the case when $900(|a|+1)d/n<1$. In this case any vertex with degree less than $d-900(|a|+1)d/n$ has degree at most $d-1$.
	Therefore, $d_B(w_1),d_B(w_2)\leq d-1$ and thus
	$$\#F(H)-\#F(H^*)
	\stackrel{\eqref{eq:bowtiediff}}{\geq} 2\frac{n^2}{4}-\frac{3n^2}{8}-101(|a|+1)dn>0$$
	where the last inequality follows because due to our condition, $101(|a|+1)q< n^2/8$, resulting in a contradiction.

	Next we consider the case when $900(|a|+1)d/n\geq 1$. Recall that by our assumption $d_B(w_1),d_B(w_2)< d-900(|a|+1)d/n$ and thus
	$$\#F(H)-\#F(H^*)
	\stackrel{\eqref{eq:bowtiediff}}{\geq} 450(|a|+1)dn-\frac{3n^2}{8}-101(|a|+1)dn>0,$$
	where the last inequality follows because due to our conditions $900(|a|+1)dn\geq n^2$, leading to a contradiction.
	
	Now we consider the remaining cases. An analogous proof works if $u,v\in V_2$.
	Should $u$ and $v$ be in different parts, by
	\eqref{eq:degdiff} we have for all $w\in V$ that $d_B(w)\geq d$ giving the required bound. 
	Also the statement is trivial when $d+1=1$, as the degree of any vertex is non-negative. Finally, if more than two vertices of degree $d+2$ exist, then $u$ and $v$ can be replaced by two vertices of degree $d+2$ and our earlier argument implies the result.
\end{proof}

\begin{proof}[Proof of Lemma~\ref{lem:degreeposs}]
	Note that $dn\leq 2(q+1) \leq (d+1)n$. We start by showing that in the graph $(V,B)$ at most one vertex of degree at least $d+2$ exists. 
	Suppose that this is false.
	Let $W_{d+1}$ be the set of vertices $w\in V$ with $d_B(w)<d+1$. By Theorem~\ref{thm:subgraph} we have $|a|\leq 1$, which together with Lemma~\ref{lem:nosmalldeg}, $d=o(n)$ and the fact that $d_B(w)$ is an integer imply that $|W_{d+1}|\leq 1$. In fact, by \eqref{eq:degdiff} we have that if there exists a vertex $w\in W_{d+1}$, then $d_B(w)\geq d$. Therefore,
	$$2(q+1)=\sum_{w\in V}d_B(w)\geq 2(d+2)+d+(n-3)(d+1)>(d+1)n,$$
	contradicting our earlier observation.
	
	Note that if there exists in $(V,B)$ a vertex of degree at least $d+3$, then by \eqref{eq:degdiff} there exists a pair of vertices $u,v\in V$ such that $d_B(u)=d_B(v)\geq d+2$, a contradiction. Therefore, there is at most one vertex with degree larger than $d+1$ in $(V,B)$ and by \eqref{eq:degdiff} this vertex, if it exists, has degree exactly $d+2$ in $(V,B)$. It only remains to show that there is at most one vertex with degree less than $d$ in this graph and should such a vertex exist it has degree $d-1$.
	
	Lemma~\ref{lem:nosmalldeg} implies that this is in fact true if $(V,B)$ contains two vertices of degree $d+1$. Now assume that $(V,B)$ has at most one vertex of degree $d+1$. By \eqref{eq:degdiff} the presence of a vertex of degree $d+2$ in $(V,B)$ would imply that many vertices of degree $d+1,d+2$ or $d+3$ should be present. However, by our assumption there is only one vertex of degree $d+1$ in $(V,B)$ and previously we have shown that there is at most one vertex of degree at least $d+2$ in $(V,B)$. Thus $(V,B)$ contains no vertex with degree larger than $d+1$. Let $W_d$ be the set of vertices $w\in V$ with $d_B(w)<d$. Then
	$$dn\leq 2(q+1)=\sum_{w\in V}d_B(w)= d+1+(n-|W_d|-1)d+\sum_{w\in W_d}d_B(w).$$
	Rearranging the terms gives us
	$$ |W_d|d-1\leq \sum_{w\in W_d}d_B(w) \leq |W_d|(d-1),$$
	where the right hand inequality follows from the definition of $W_d$. The inequality holds only if $|W_d|\leq 1$ and for $w\in W_d$ we have $d_B(w)=d-1$.
\end{proof}

\begin{proof}[Proof of Lemma~\ref{lem:partitiondegree}]
	Assume for contradiction that the size of each of these sets is at least two. 
	Therefore there exist vertices $u_{1},u_{2}\in C_{k+1},v_{1},v_{2}\in C_{k},w_{1},w_{2}\in D_{\ell+1},z_{1},z_{2}\in D_{\ell}$.  
	Roughly speaking, we want to show that either moving an edge from $C_{k+1}$ to $D_{\ell}$ decreases the number of bowties, or moving an edge from $D_{\ell+1}$ to $C_{k}$ does.
	
	Define $\phi_1:V\rightarrow\mathbb{N}$ by
		$$	\phi_1(x) =
		\left\{
		\begin{array}{ll}
		d_B(x)-1,  & \mbox{if }  x=u_1,u_2,\\
		d_B(x), & \mbox{if } v\in V\setminus\{u_1,u_2,z_1,z_2\},\\
		d_B(x)+1, & \mbox{if } x=z_1,z_2,
		\end{array}
		\right.
		$$	
	and
	$\phi_2:V\rightarrow\mathbb{N}$ by
	$$	\phi_2(x) =
	\left\{
	\begin{array}{ll}
	d_B(x)-1,  & \mbox{if }  x=w_1,w_2,\\
	d_B(x), & \mbox{if } v\in V\setminus\{w_1,w_2,v_1,v_2\},\\
	d_B(x)+1, & \mbox{if } x=v_1,v_2.
	\end{array}
	\right.
	$$	
		
	For $i=1,2$ let $H_i$ be the graph returned by Lemma~\ref{lem:trifree} with $\phi=\phi_i$, $V_1^*=V_1$ and $V_2^*=V_2$.
 Thus, by \eqref{eq:numbowties}, we have
	$$\#F(H_1)-\#F(H)\le 2\ell\partition{1} (\partition{1} -2)+(\badedgep{1}-1)(2n-8)+\badedgep{2}\partition{1} -2k\partition{2} (\partition{2} -2)-\badedgep{2}(2n-8)-(\badedgep{1}-1)\partition{2}$$
and
	$$\#F(H_2)-\#F(H)\le 2k\partition{2} (\partition{2}-2)+(\badedgep{2}-1)(2n-8)+\badedgep{1}\partition{2}-2\ell\partition{1} (\partition{1} -2)-\badedgep{1}(2n-8)-(\badedgep{2}-1)\partition{1}.$$
	Should either one of these be negative, we are done, as we have a contradiction on the minimality of $H$. Clearly, this holds if the sum of the two terms is negative.  Note that the $2\ell\partition{1} (\partition{1}-2)$ and the $2k\partition{2}(\partition{2}-2)$ terms cancel and thus
	\begin{align*}
	\#F(H_1)+\#F(H_2)-2\#F(H)&\le (\badedgep{1}-1)(2n-8)+\badedgep{2}\partition{1}-\badedgep{2}(2n-8)-(\badedgep{1}-1)\partition{2}\\
	&+(\badedgep{2}-1)(2n-8)+\badedgep{1}\partition{2} -\badedgep{1}(2n-8)-(\badedgep{2}-1)\partition{1}\\
	&=-(2n-8)+\partition{1}-(2n-8)+\partition{2}\\
	&=-3n+16<0
	\end{align*}
	and the statement follows.
\end{proof}

\section{Concluding remarks}\label{conclusion}

After Proposition~\ref{lem:cbg} is established (namely, that every extremal graph admits a vertex partition $\{V_1,V_2\}$ such that all cross edges are present), the rest essentially reduces to the problem of minimising the right-hand side of~\eqref{eq:numbowties} as a function of e.g.\ $v_1=|V_1|$ and $b_1=|E(H(V_1))|$. Surprisingly, this integer optimisation problem turned out to be very delicate and we needed a lot of calculations to solve it  (i.e.\ to derive Theorems~\ref{thm:subgraph}--\ref{thm:assymptotics} from Proposition~\ref{lem:cbg}). While our method may apply to other non-critical graphs $F$, we expect that similar algebraic difficulties  will appear. One important difference between the cases of critical and non-critical $F$ for $q=o(n^2)$ is that in the former case a single edge added to $K(V_1,V_2)$ already creates many copies of $F$, which often gives the dominant term for $h_F(n,q)$ and makes analysis easier.

One cannot expect that Theorem~\ref{thm:subgraph} holds for all $q$. For example, the Tur\'an graph $T_r(n)$ is known to minimise the number of triangles among all graphs of the same order and size, which follows from the results by 
Moon and Moser~\cite{MoonMoser62} and, independently, Nordhaus and Stewart \cite{NordhausStewart63}, which can be also derived from the paper of Goodman~\cite{Goodman59}. The corresponding stability result was obtained by Lov\'{a}sz and Simonovits \cite{lovasz+simonovits:83}. Since this graph has approximately the same number of triangles per each vertex, the Cauchy-Schwarz Inequality implies that $T_r(n)$ also asymptotically minimises the number of bowties, including the stability result that all asymptotically optimal graphs are $o(n^2)$-close  in the edit distance to $T_r(n)$. However, if $r\ge 3$ is a fixed odd integer, then $T_r(n)$ is $\Omega(n^2)$-away from containing $T_2(n)$. So Theorem~\ref{thm:subgraph} strongly fails  for the corresponding value of~$q$.
It would be interesting to know how $h_F(n,q)$ behaves for larger $q$, in particular find the largest $c$ such that $h_F(n,q)=(1+o(1))t_F(n,q)$ holds for every $q\leq cn^2$. 

\bibliography{bowtie}
\bibliographystyle{plain}

\end{document}